\documentclass[12pt,reqno,letterpaper]{amsart}
\usepackage{mathrsfs}

\usepackage[T1]{fontenc}
\usepackage{enumitem}
\usepackage{graphicx}
\graphicspath{ {./images/} }
\usepackage{esint}
\usepackage{hyperref}
\hypersetup{
  colorlinks   = true,
  citecolor    = blue,
  linkcolor    = blue
}

\usepackage{amsmath,amsfonts,amsthm,color,tikz, comment, xcolor, xfrac,amssymb}

\usepackage[normalem]{ulem}
\usepackage{mathrsfs}

\usepackage{pdfsync}
\usepackage[font={scriptsize}]{caption}
\usepackage{environ}

\usepackage[left=1in, right=1in, top=1in,bottom=1in]{geometry}
\numberwithin{equation}{section}

\def\XXint#1#2#3{{\setbox0=\hbox{$#1{#2#3}{\int}$ }
\vcenter{\hbox{$#2#3$ }}\kern-.6\wd0}}

\usepackage{fix-cm}

\newcommand{\R}{\mathbb R}
\newcommand{\N}{\mathbb N}
\newcommand{\Q}{\mathbb Q}
\newcommand{\T}{\mathbb T}
\newcommand{\Z}{\mathbb Z}

\newtheorem{theorem}{Theorem}[section]

\newtheorem{definition}[theorem]{Definition}

\newtheorem{lemma}[theorem]{Lemma}

\newtheorem{proposition}[theorem]{Proposition}

\theoremstyle{remark}
\newtheorem{remark}[theorem]{Remark}

\theoremstyle{remark}

\newcommand{\bean}{\begin{eqnarray*}}
\newcommand{\eean}{\end{eqnarray*}}
\newcommand{\ben}{\begin{enumerate}}
\newcommand{\een}{\end{enumerate}}
\newcommand{\beq}{\begin{equation}}
\newcommand{\eeq}{\end{equation}}

\allowdisplaybreaks

\begin{document}

\author{Nicholas Gismondi}

\title{Nontrivial Integrable Weak Stationary Solutions to Active Scalar Equations with Non-Odd Drift}

\begin{abstract}
    In this paper we construct nontrivial weak solutions to a class of stationary active scalar equations with a non-odd nonlocal operator in the drift term using a convex integration scheme. We show our solutions lie in
    $$
    \bigcap_{0 < \epsilon < 1} \dot{B}^{-\epsilon}_{\infty,\infty}(\T^d) \cap L^{2-\epsilon}(\T^d)
    $$
    for $d \geq 2$. The key ingredient of the construction is the use of highly oscillatory corrections with a variable degree of intermittency, which is arranged to decrease to zero at higher stages of the iteration procedure.
\end{abstract}

\maketitle

\section{Introduction}

\subsection{Motivation and Background}
We consider the stationary active scalar equation:
\begin{equation}\label{eq:stat_eqn}
\begin{cases}
    \nabla \cdot (\theta u) + \Lambda^\gamma \theta = 0\\
    u = T\theta
\end{cases}
\end{equation}
where $\theta:\T^d \to \R$, $u:\T^d \to \R^d$, $d \geq 2$, $\Lambda^\gamma$ is the fractional Laplacian (see Definition \ref{def:laplace}) with $0 < \gamma \leq 2$, and $T$ is a convolution type Calderon-Zygmund operator which sends $\R$ valued functions to divergence free $\R^d$ valued ones and whose multiplier $m:\Z^d \to \R^d$ is not odd. We also assume that $m$ is homogeneous of order $0$, i.e. $m(\lambda \xi) = m(\xi)$ for all $\lambda > 0$ and $\xi \in \Z^d$, and coincides with a smooth function away from the origin.

The restriction to multipliers $m$ which are not odd is essential for the convex integration scheme employed in this paper. Indeed, the aim is to decompose the nonlinear term into its mean and projection off the mean, and then use the mean to cancel leading order errors generated by the highly oscillatory perturbations, which we denote $w:\T^d \to \R$. If $m$ was odd, this would force $T$ to be skew-adjoint on $L^2(\T^d)$, and hence we would have that
$$
\int_{\T^d} w Tw = 0
$$
for $w$ smooth. So in the case when the multiplier is odd, it becomes impossible to cancel the leading order errors using the mean of the nonlinear term.

To overcome this obstruction, one must instead assume the error term can be written as the divergence of a symmetric matrix, and then "invert the divergence." While this approach is feasible for certain specific odd multipliers, it appears to be beyond the reach of current convex integration methods for general skew-adjoint Calderon–Zygmund operators. We refer the reader to the discussion below concerning the SQG equation for an example of a situation in which this strategy is feasible.

Having clarified the structural restriction imposed on the  multipliers, we now briefly recall the corresponding non-stationary active scalar equation. The non-stationary active scalar equation, or just active scalar equation for simplicity, is given by
\begin{equation}\label{eq:ase}
    \begin{cases}
        \partial_t \theta + \nabla \cdot (\theta u) + \Lambda^\gamma \theta = 0\\
        u = T\theta
    \end{cases}
\end{equation}
where $\theta: [0,T) \times \T^d \to \R$ and $u:[0,T) \times \T^d \to \R$. The active scalar equation is said to be \textit{inviscid} if the $\Lambda^\gamma \theta$ term in ~\eqref{eq:ase} is removed. By convention when $\gamma = 0$ this refers to the inviscid case. Active scalar equations arise frequently in physical contexts, and we list three examples here:
\begin{enumerate}
    \item [(a)] The incompressible porous media (IPM) equation has multiplier give by
    $$
    m(\xi) = \frac{\langle \xi_1 \xi_2, -\xi_1^2\rangle}{|\xi|^2}
    $$
    in two dimensions while in three dimensions it is given by
    $$
    m(\xi) = \frac{\langle \xi_1\xi_3, \xi_2\xi_3, -\xi_1^2 - \xi_2^2\rangle}{|\xi|^2}.
    $$
    Notice this multiplier is even, bounded, homogeneous of order $0$, and smooth away from the origin, and therefore our proof will apply to both two and three dimensional stationary IPM.
    \item [(b)] The magneto-geostrophic (MG) equation is a three dimensional active scalar equation which has multiplier given by
    $$
    m(\xi) = 
    \begin{cases}
        \frac{\left\langle \xi_2\xi_3|\xi|^2 + \xi_1\xi_2^2 \xi_3, -\xi_1\xi_3|\xi|^2 + \xi_2^3 \xi_3, -\xi_2^2(\xi_1^2 + \xi_2^2)\right\rangle}{\xi_3^2|\xi|^2 + \xi_2^4}, & \xi_3 \not=0\\
        0, & \xi_3 = 0
    \end{cases}.
    $$
    We refer to \cite{FRV} and \cite{LM} for more information on the physical relevance of this equation. For us, there are two issues with this multiplier. The more benign issue is that it is not smooth away from the origin. This can be handled quite easily, see \ref{footnote}. The more serious issue is that this multiplier is not bounded. Indeed, observe that $|m(\lambda^2,\lambda,1)| \simeq \lambda^2$. This means the convolution type operator in the MG equation is not of Calderon-Zygmund type, and is thus out of reach of our method. Since the convex integration schemes employed in \cite{IV} and \cite{Shvydkoy} apply to both the IPM equation and the MG equation, there is some hope that either the method here can be modified or a completely new method may be utilized to handle operators not of Calderon-Zygmund type. We leave this as an interesting open question worthy of further study.
    \item [(c)] The surface quasi-geostrophic (SQG) equation is a two dimensional active scalar equation which has its multiplier given by
    $$
    m(\xi) = i\frac{\langle \xi_2,-\xi_1\rangle}{|\xi|}.
    $$
    We refer to \cite{Salmon} for a discussion of the physical application of the SQG equation. Notice the multiplier here is odd, so our result does not apply to this equation. See \cite{CKL} and \cite{GR} for treatments of stationary SQG, and \cite{BSV}, \cite{DGR}, and \cite{IL} for non-stationary SQG; the latter two of which resolve the Onsager conjecture for the inviscid SQG equation.
\end{enumerate}
In this paper we utilize the method of convex integration, first introduced by De Lellis and Sz\'{e}kelyhidi in \cite{DLS09} and \cite{DLS13}, to construct nontrivial weak solutions to ~\eqref{eq:stat_eqn} which lie in
$$
\theta,u \in \bigcap_{0 < \epsilon < 1} \dot{B}^{-\epsilon}_{\infty,\infty}(\T^d) \cap L^{2-\epsilon}(\T^d).
$$
The method of convex integration has been successfully applied to stationary partial differential equations which arise from fluid dynamics considerations before. For instance the method was used to construct nontrivial weak solutions to the stationary Navier-Stokes equations in dimensions 2 \cite{ABGN} and $d \geq 4$ \cite{Luo}. The solutions in \cite{Luo} are $L^2(\T^d)$, whereas in \cite{ABGN} they are only able to build solutions which lie in $\cap_{\epsilon > 0} L^{2-\epsilon}(\T^2)$. In a very similar procedure to \cite{Luo}, in \cite{Peng} $L^2(\T^3)$ solutions to the stationary electron MHD equation were constructed. In addition, as already noted, the method has successfully been applied to the stationary SQG equation to produce solutions which lie in $\dot{B}^\alpha_{\infty,\infty}(\T^2)$ for $-1/2 \leq \alpha \leq -1/2 + \min(1/6,3/2-\gamma)$ \cite{CKL} and in \cite{GR} it was utilized to produce nontrivial solutions which lie in $L^p(\T^2)$, $p < 4/3$.

In \cite{DW}, the same question is considered, but the authors are only able to construct solutions which lie in $\dot{B}^{\alpha - 1}_{\infty,\infty}(\T^d)$ for fixed $\alpha < 1$ and $0 < \gamma < 2 - \alpha$. The mechanism which allows us to also deduce the integrability is the use of intermittent building blocks in the construction of the highly oscillatory increments, see ~\eqref{eq:wq+1}. Intermittency has played a key role in many convex integration constructions; see for instance  \cite{ABGN}, \cite{BBV}, \cite{BCV}, \cite{BMNV}, \cite{BV19}, \cite{CL2}, \cite{CheskidovLuo}, \cite{DS17}, \cite{GR}, \cite{Luo}, \cite{MS}, \cite{NV}, \cite{Peng} and references therein.

It should be noted that as in \cite{DW}, our solutions unfortunately do not lie in $L^2(\T^d)$. This means that $\theta u$ is not $L^1(\T^d)$, and so an alternative notion of weak solution is required. One potential solution is provided by \cite{PGLR} where the pair $\theta$ and $u$ is said to be admissible if the term-by-term product of their Fourier series is absolutely convergent in some Sobolev space. We instead consider a variant of this idea from \cite{ABGN} and \cite{GR} to define $\mathbb{P}_{\not=0}(\theta u)$ as an element of $\dot{H}^{-s}(\T^d)$ for some $s$ chosen large enough if the term-by-term product of their paraproduct expansions is absolutely convergent in $\dot{H}^{-s}(\T^d)$. We then use dual pairings to define the notion of a weak solution; see Definitions \ref{def:paras} and \ref{def:weak_para_soln}. As was noted in \cite{DW}, the existence of nontrivial weak solutions to ~\eqref{eq:stat_eqn} which lie in $L^2(\T^d)$ seems to be a difficult open problem.

\subsection{Main Result}

As mentioned already, our first task is to extend the definition of a weak solution to ~\eqref{eq:stat_eqn} to $u$ and $\theta$ lying in some Sobolev space with strictly negative regularity. To do this, we need to make sense of the product $\theta u$. In general, there is not much that can be said about this product when both functions lie in a Sobolev space with negative regularity, but following \cite[Definition 1.1]{ABGN} we may offer the following definition of the mean free product:

\begin{definition}[\textbf{Paraproducts in $\dot H^s(\T^d)$}]\label{def:paras}
    Let $f,g$ be distributions, so that $\mathbb{P}_{2^j}(f), \mathbb{P}_{2^{j'}}(g)$ are well-defined for $j,j'\geq 0$ (see Definition \ref{def:projs}).  We say that $\mathbb{P}_{\not = 0}(fg)$ is well-defined as a paraproduct in $\dot{H}^s(\T^d)$ for some $s \in \R$ if
    $$
        \sum_{j,j' \geq 0} \left\Vert \mathbb{P}_{\not = 0}\left(\mathbb{P}_{2^j}(f) \mathbb{P}_{2^{j'}}(g)\right) \right\Vert_{\dot{H}^s} < \infty \, .
    $$
    Then we define
    $$
        \mathbb{P}_{\not = 0}(fg) = \sum_{j,j' \geq 0} \mathbb{P}_{\not = 0} \left(\mathbb{P}_{2^j}(f) \mathbb{P}_{2^{j'}}(g)\right) \, ,
    $$
    since the right-hand side is an absolutely summable series in $\dot H^s(\T^d)$.
\end{definition}

With this definition, adapting \cite[Definition 1.2]{ABGN}, we may now define a weak solution to \eqref{eq:stat_eqn} which is valid for $u$ and $\theta$ belonging to Sobolev spaces of arbitrary regularity.

\begin{definition}[\textbf{Weak paraproduct solutions to ~\eqref{eq:stat_eqn}}] \label{def:weak_para_soln}
    If $\theta \in \dot{H}^s$, $s < 0$, and $u = T \theta \in \dot{H}^s$, we say $u$ and $\theta$ form a weak paraproduct solution to the stationary active scalar equation if there is $s' \in \R$ such that $\mathbb{P}_{\not = 0}(\theta u)$ is well defined as a paraproduct in $\dot{H}^{s'}$ in the sense of the previous definition and
    $$
        \left\langle \theta, \Lambda^{\gamma} \phi \right\rangle_{\dot{H}^{s},\dot{H}^{-s}} - \left\langle \mathbb{P}_{\not= 0}(\theta u), \nabla \phi \right\rangle_{\dot{H}^{s'}, \dot{H}^{-s'}} = 0
    $$
    for all smooth $\phi$.
\end{definition}

With this, we may state our main result.

\begin{theorem}\label{thm:main}
    Given any $0 < \gamma \leq 2$ there exist
    $$
    u,\theta \in \bigcap_{0 < \epsilon < 1} \left(\dot{B}_{\infty,\infty}^{-\epsilon}(\T^d) \cap L^{2-\epsilon}(\T^d)\right) \setminus \{0\}
    $$
    such that
    \begin{enumerate}
        \item [(a)] $u = T\theta$;
        \item [(b)] $\mathbb{P}_{\not=0}(\theta u)$ exists as a paraproduct in $\dot{H}^{-s}(\T^d)$ for all $s > d/2 + \max(\gamma - 1, 0)$;
        \item [(c)] $u$ and $\theta$ solve Equation ~\eqref{eq:stat_eqn} in the sense of Definition \ref{def:weak_para_soln}.
    \end{enumerate}
\end{theorem}

\begin{remark}
    As in \cite{GR}, this appears to be the first construction of weak solutions to ~\eqref{eq:stat_eqn} on the torus which are \textit{integrable functions} and not merely \textit{distributions}.
\end{remark}

\begin{remark}
    Since our solutions $u,\theta \in L^p(\T^d)$ for all $p < 2$, from Sobolev embedding we also see that $u,\theta \in \dot{H}^{-\epsilon}(\T^d)$ for all $\epsilon > 0$.
\end{remark}

\begin{remark}
    In \cite{DW}, the authors use the following notion of a weak solution to Equation ~\eqref{eq:stat_eqn}: if for all smooth $\psi:\T^d \to \R$ with compact frequency support we have that
    \begin{equation}\label{eq:DW_soln_def}
        - \int_{\T^d} \theta u \cdot \nabla \psi + \int_{\T^d} \theta \Lambda^\gamma \psi = 0
    \end{equation}
    then we call the pair $u$ and $\theta$ a weak solution. Notice if we restrict our test functions to those which have compact frequency support, then we attain essentially the same definition since then using the self-adjointness of frequency truncations, we pass the smooth truncation onto $\mathbb{P}_{\not=0}(\theta u)$, and since this is well defined as an element of $\dot{H}^{-s}$, the frequency truncation will be smooth and in particular $L^1$. Therefore solutions in the sense of Definition \ref{def:weak_para_soln} will also be a solution in the sense of \cite[Definition 1.1]{DW}.
\end{remark}

\subsection{Outline of Paper}
In Section \ref{Section2} we review some material related to Littlewood-Paley theory and the function spaces we will be considering in this paper. We also introduce the intermittent building block we utilize in the construction. In Section \ref{Section3} we give a brief overview of the convex integration scheme to follow. In Section \ref{Section4} we state the main inductive proposition (Proposition \ref{prop:ind}) and use it to prove Theorem \ref{thm:main}. We then provide the definition of our highly oscillatory perturbation and use it to prove Proposition \ref{prop:ind} in Section \ref{Section5}.

\subsection{Acknowledgments} I am deeply thankful to Alexandru Radu for discussion in relation to \cite{GR}.

\section{Background Theory and Technical Lemmas}\label{Section2}

The following on Littlewood-Paley theory, Fourier multiplier operators, and function spaces can be found in \cite{BCD} and \cite{Grafakos}. The exact formulation of Definition \ref{def:projs} is based off \cite[Definition 2.1]{ABGN} and \cite[Equation 4.9]{BSV}.

\begin{definition}[\textbf{Littlewood-Paley projectors}]\label{def:projs}
    There exists $\varphi \colon \R^{d}\to[0,1]$, smooth, radially symmetric, and compactly supported in $\{\xi : 6/7 \leq |\xi|\leq 2\}$ such that $\varphi(\xi) = 1$ on $\{\xi : 1 \leq |\xi| \leq 12/7\}$,
    \begin{equation}
        \sum_{j\geq 0}\varphi(2^{-j}\xi)=1 \hspace{0.25cm} \text{ for all } \hspace{0.25cm} |\xi|\geq 1, \notag
    \end{equation}
    and $\operatorname{supp}\varphi_{j}\cap \operatorname{supp} \varphi_{j'}=\emptyset$ for all $|j-j'|\geq 2$, where $\varphi_j(\cdot) = \varphi(2^{-j}\cdot)$. We define the projection of a function $f$ on its $0$-mode by
    \begin{equation}
        \mathbb{P}_{=0}(f)=\int_{\T^d} f(x)\, dx, \notag
    \end{equation}
    and the projection on the $j^{\rm th}$ shell by
    \begin{equation}
        \mathbb{P}_{2^j}(f)(x)=\sum_{k\in \Z^d}\hat{f}(k)\varphi_{j}(k)e^{2\pi ik\cdot x} \, . \notag
    \end{equation}
    We also define $\mathbb{P}_{\neq 0}f:=(\operatorname{Id}-\mathbb{P}_{=0})f$, and we also denote by $\hat{K}_{\simeq 1}$ a smooth radially symmetric bump function with support in the ball $\{\xi: |\xi| < r\} $, which also satisfies $\hat{K}_{\simeq 1}(\xi)=1$ on the smaller ball $\{\xi: |\xi| \leq \frac{1}{2}r\}$ for $r \ll 1$ is a power of $2$ whose exact size will be specified in Section \ref{sec:param}. Then let $\mathbb{P}_{\leq \lambda}$ be the convolution operator that has $\hat{K}_{\simeq 1}\left(\frac{\xi}{\lambda}\right)$ as its Fourier multiplier.
\end{definition}

The following lemma is standard, and can be deduced as a consequence of the Poisson summation formula.

\begin{lemma}[\textbf{$L^p$ boundedness of projection operators}]\label{lem:proj}
    $\mathbb{P}_{\leq \lambda}$ is a bounded operator from $L^p$ to $L^p$ for $1 \leq p \leq \infty$ with operator norm independent of $\lambda$.
\end{lemma}

\begin{definition}[\textbf{Fractional Laplacian}]\label{def:laplace}
    For $u:\T^d \to \R$ and $s \in \R$, define the fractional Laplacian operator $\Lambda^s = (-\Delta)^{s/2}$ by
    $$
        \left(\Lambda^su\right)^{\wedge}(\xi) = |2\pi \xi|^s \hat{u}(\xi) \quad \forall \xi \in \Z^d.
    $$
    When $s < 0$ we restrict the above definition to just those $\xi \not = 0$.
\end{definition}

Next we introduce the function spaces which will play a prominent role throughout.

\begin{definition}[\textbf{$\dot H^s$ Sobolev spaces}] For $s \in \mathbb{R}$, we define
    $$\dot{H}^{s}(\mathbb{T}^{d})=\left\{ f: \sum_{\xi \in \Z^d \setminus \{0\}}|\xi|^{2s}|\hat{f}(\xi)|^2<\infty\right\}  $$
    with the norm induced by the sum above.
\end{definition}
\begin{remark}
    For every $f\in \dot H^{s}$ for some $s\in \R$, we define the Fourier coefficients
    \[\hat{f}(\xi)=\int_{\T^{d}}f(x)e^{-2\pi i\xi\cdot x}\,dx, \hspace{0.25cm} \text{ where } \hspace{0.25cm} \T^{d}=[0,1]^{d},\]
    and so we can define $\mathbb{P}_{2^j}f$ for $j\geq 0$.  Note that each $\mathbb{P}_{2^j}f$ is smooth if $f\in \dot H^{s}$ irrespective of the value of $s\in \R$.
\end{remark}
\begin{remark}
    For $s > 0$ the following equivalent definition of the $\dot{H}^{-s}$ norm given by
    $$
        \Vert f \Vert_{\dot{H}^{-s}} = \sup_{\Vert \phi \Vert_{\dot{H}^s = 1}} \left|\left\langle f,\phi \right\rangle_{\dot{H}^{-s},\dot{H}^s}\right|
    $$
    is available. With this definition, it is easy to see that the $\dot{H}^{-s}$ norm can be pushed inside integrals, which we perform quite often in Section \ref{Section5} without comment.
\end{remark}

\begin{definition}{\textbf{(Homogeneous Besov spaces)}}
    For $\alpha \in \R$, $0<p,q\leq \infty$, and $f \in \mathcal{D}'(\T^d)$, define the homogeneous Besov space to be
    \[
        \dot{B}^{\alpha}_{p,q} = \left\{ f: \|f\|_{\dot{B}^{\alpha}_{p,q}} = \left\| 2^{j\alpha}\|\mathbb{P}_{2^j}f\|_{L^p(\T^d)}\right\|_{\ell^q_{j\geq 0}} <\infty \right\}.
    \]
\end{definition}
% \vspace{-2.5em}
\begin{remark}
    Note for $\alpha > 0$ the space $\dot{B}^{\alpha}_{\infty,\infty}$ is equivalent to $\{f \in C^\alpha : \hat{f}(0) = 0\}$. In addition, the space $\dot{B}^{\alpha}_{2,2}$ is equivalent to the homogeneous Sobolev space $\dot{H}^\alpha$ for all $\alpha \in \R$.
\end{remark}

One of our first tasks will be to establish an analogue of "reconstruction lemmas" that are frequently utilized in convex integration constructions. The purpose of this is in our oscillation error estimates we will need to cancel to what we call the stress field from the previous step of the iteration. This is accomplished by considering quadratic expressions involving our Nash iterates, extracting a low frequency term, and then utilizing this reconstruction lemma (Lemma \ref{lem:lin_alg}) to achieve the desired cancellation. However, as we will see, the vectors available to us to perform this cancellation procedure lie in the image of $m(\xi) + m(-\xi)$. This means we first must ensure that we are able to extract a basis from the image of $m(\xi) + m(-\xi)$, and in particular this basis we extract must satisfy certain properties that will prove useful to us in the analysis. In two dimensions, this can be done. However, as was noted in \cite{IV}, in dimensions three and higher this must instead be taken as an assumption. For two dimensions, the existence of such a basis as well as all of the properties we will utilize is the content of Lemmas \ref{lem:NT} and \ref{lem:nondegen}.

\begin{lemma}[\textbf{Lattice vectors within a given arc}]\label{lem:NT}
Fix an open arc $U \subset \mathbb{S}^1$. Then there are $v_1,v_2 \in \Z^2$ such that $\{v_1,v_2\}$ form a basis of $\R^d$, $|v_1| = |v_2| \in \Z$, and $v_1/|v_1|, v_2/|v_2| \in U$.
\end{lemma}
\begin{proof}
Since rational points are dense in the unit circle, we pick two points $u_1,u_2 \in U \cap \Q^2$ with $u_1$ and $u_2$ not collinear. Since $u_1$ and $u_2$ have rational coordinates, there is $N \in \Z$ such that $Nu_1,Nu_2 \in \Z^d$. We pick $v_1 = Nu_1$ and $v_2 = Nu_2$ and claim all desired properties are satisfied. First since $u_1$ and $u_2$ were chosen not to be collinear, $v_1$ and $v_2$ are not as well. This is enough to ensure linear independence. Next we have
$$
|v_1| = N|u_1| = N = N|u_2| = |v_2|
$$
and since $N \in \Z$ the second condition is satisfied. Then finally
$$
\frac{v_1}{|v_1|} = u_1 \in U.
$$
The same argument shows $v_2/|v_2| \in U$ completing the proof.
\end{proof}

\begin{lemma}[\textbf{Non-degeneracy of multiplier in two dimensions}]\label{lem:nondegen}
    There are $k_1,k_2 \in \Z^2$ with the same norm such that $\{m(k_1) + m(-k_1), m(k_2) + m(-k_2)\}$ is a basis of $\R^2$.
\end{lemma}
\begin{proof}
    First, note that from the assumption that $T$ sends $\R$ valued functions to $\R^2$ valued ones, the multiplier must satisfy
    $$
    m(\xi) = \overline{m(-\xi)}
    $$
    for all $\xi \in \Z^2$ and thus it follows that
    $$
    m(\xi) + m(-\xi) \in \R^2
    $$
    for all $\xi \in \Z^2$. Let $\tilde{m}$ denote the smooth function which agrees with $m$ away from the origin. Since $m$ is not odd, there is $\xi \in \mathbb{S}^1$ such that
    $$
    \tilde{m}(\xi) + \tilde{m}(-\xi) \not = 0.
    $$
    Then there is an open arc $U \subset \mathbb{S}^1$ which contains $\xi$ such that
    $$
    \tilde{m}(\eta) + \tilde{m}(-\eta) \not=0 
    $$
    for all $\eta \in U$. Applying Lemma \ref{lem:NT}, we may find $\{k_1,k_2\} \subset \Z^2$ such that $|k_1| = |k_2|$, $\{k_1,k_2\}$ are linearly independent, and $k_1/|k_1|, k_2/|k_2| \in U$. Thus we have that
    $$
    m(k_1) + m(-k_1) \not = 0
    $$
    and
    $$
    m(k_2) + m(-k_2) \not= 0
    $$
    using the homogeneity of $m$. Since $\{k_1,k_2\}$ is a basis of $\R^d$, $\{k_1^\perp,k_2^\perp\}$ are as well. Then finally since $T$ produces divergence free vector fields, $m(k_1) + m(-k_1)$ is a non-zero multiple of $k_1^\perp$ and $m(k_2) + m(-k_2)$ is a non-zero multiple of $k_2^\perp$. This completes the proof.
\end{proof}

\begin{lemma}[\textbf{Linear algebra}]\label{lem:lin_alg}
    If $d = 2$, let $\Omega \subset \Z^d$ be the basis constructed in Lemma \ref{lem:nondegen}. If $d \geq 3$, then we assume there is a basis $\Omega \subset \Z^d$ satisfying the conclusions of Lemma \ref{lem:nondegen} adapted to $d$ dimensions. Put
    $$
    k^* = \sum_{k \in \Omega} \left(m(k) + m(-k)\right) \not = 0.
    $$
    For $\epsilon_{\Omega} > 0$ small enough there exist smooth functions $\{\Gamma_k : \R^d \to \R\}_{k \in \Omega}$ such that for all $v \in B(k^*,\epsilon_{\Omega})$ we have
        $$
        v = \sum_{k \in \Omega} \left(\Gamma_k(v)\right)^2 \left(m(k) + m(-k)\right);
        $$
\end{lemma}
\begin{proof}
This is a very simple application of the inverse function theorem. See \cite[Proposition 5.6]{BV2020} and \cite[Lemma 3.2]{DLS13} for more details related to the case where the vector space under consideration is the space of three by three symmetric matrices; the proof in this case is a slight variant. The idea here is very reminiscent of what is utilized in \cite[Section 2.1]{IV} and \cite[Lemma 2.1]{DW}.
\end{proof}

One of the advantages of estimating our stress field in a Sobolev norm with negative regularity is that we never require the derivative to "land" on our high frequency building blocks. This means that we do not require that these high frequency building blocks are a stationary solution to any PDE. In \cite{GR}, this was exploited to utilize a fully intermittent building block, which typically would not be acceptable in convex integration schemes where the Reynolds stress or stress field is estimated in an $L^p$ or $C^0$ norm. In this work, we will not require a fully intermittent building block; one dimension of intermittency will suit our purposes. This means that if one wishes to use a more standard choice of high frequency building block, like an intermittent Mikado flow in \cite{ABGN}, that would be perfectly acceptable. However, as in \cite{ABGN} and \cite{GR}, our analysis will again not require the building block solve any PDE, so to demonstrate the flexibility that this method affords, we intentionally use a generic high frequency building block which we call an intermittent slab.

\begin{lemma}[\textbf{Intermittent slabs}]\label{lem:boldW}
    Let $\lambda_{q+1}$ be a large integer and take $0 < \epsilon_{q+1} < 1$ such that $\lambda_{q+1}^{\epsilon_{q+1}}$ is also an integer. For each $k \in \Omega$ from Lemma \ref{lem:lin_alg}, there exist smooth $\rho^k_{q+1}:\T^d \to \R$ such that
    \begin{enumerate}
        \item\label{w:2} $\int_{\T^d} \rho_{q+1}^k = 0$;
        \item\label{w:3} $\Vert \nabla^\alpha \rho_{q+1}^k \Vert_{L^p(\T^d)} \lesssim \lambda_{q+1}^{(1-\epsilon_{q+1})\left(\frac{1}{2}-\frac{1}{p}\right)} \lambda_{q+1}^{|\alpha|}$;
        \item\label{w:4} $\rho^k_{q+1}$ is $\left(\frac{\T}{\lambda_{q+1}^{\epsilon_{q+1}}}\right)^d$-periodic.
    \end{enumerate}
\end{lemma}

\begin{proof}
    Let $\phi:\R \to \R$ be smooth, odd, supported in $(-1,1)$, mean-zero, and $L^2$ normalized. Define $\rho_{q+1}^k:\R^d \to \R$ by
    \begin{equation}\label{eq:rho}
        \rho^k_{q+1}(x) = \sum_{n \in \Z} \lambda_{q+1}^{\frac{1-\epsilon_{q+1}}{2}} \phi\left(\lambda_{q+1}k \cdot x +\lambda_{q+1}^{1-\epsilon_{q+1}}n\right).
    \end{equation}
    We claim this function will have all of the desired properties. Items \ref{w:2} and \ref{w:4} are obvious, leaving \ref{w:3}. For this, our strategy will be to obtain the desired estimate for $p = 1$ and $p = \infty$, and then use interpolation to obtain the desired bound for all intermediate values of $p$. First, note that taking derivatives will cost a factor of $\lambda_{q+1}$, so without loss of generality we may assume that $|\alpha| = 0$. We start with the $L^\infty$ estimate. For this, note for $\lambda_{q+1}$ large enough, the terms inside the summation in ~\eqref{eq:rho} will be disjoint. Thus
    \begin{equation}\label{eq:rho_L^inf_est}
        \Vert \rho^k_{q+1} \Vert_{L^\infty(\T^d)} \lesssim \lambda_{q+1}^{\frac{1-\epsilon_{q+1}}{2}} \Vert \phi \Vert_{L^{\infty}(\R)} \simeq \lambda_{q+1}^{\frac{1-\epsilon_{q+1}}{2}}.
    \end{equation}
    Now for the $L^1$ estimate, each term in ~\eqref{eq:rho} is supported in a rectangle of measure $\lambda_{q+1}^{-1}$. Each of these supports is contained within a rectangle of area $\lambda_{q+1}^{-\epsilon_{q+1}}$. Hence
    $$
        \left|\operatorname{supp}\left(\rho_{q+1}^k\right) \cap [0,1]^2 \right| \simeq \frac{\lambda_{q+1}^{-1}}{\lambda_{q+1}^{-\epsilon_{q+1}}} = \lambda_{q+1}^{\epsilon_{q+1}-1}.
    $$
    Utilizing this observation as well as ~\eqref{eq:rho_L^inf_est} we have
    \begin{equation}\label{eq:rho_L1_est}
        \Vert \rho_{q+1}^k \Vert_{L^1(\T^d)} \lesssim \Vert \rho_{q+1}^k \Vert_{L^\infty(\T^d)} \left|\operatorname{supp}\left(\rho_{q+1}^k\right) \cap [0,1]^2 \right| \lesssim \lambda_{q+1}^{\frac{\epsilon_{q+1}-1}{2}}.
    \end{equation}
    Interpolating \cite{Stein} between the estimates provided by ~\eqref{eq:rho_L^inf_est} and ~\eqref{eq:rho_L1_est} gives \ref{w:3} and completes the proof.
\end{proof}

\begin{lemma}[\textbf{Fourier series representation of the intermittent slab}]\label{lem:fourier}
    The Fourier series representation of $\rho_{q+1}^{k}$ is
    \begin{equation}\label{eq:fourier_series}
        \rho_{q+1}^{k}(x) = \sum_{n \in \Z} \lambda_{q+1}^{\frac{\epsilon_{q+1}-1}{2}} \hat{\phi}\left(\lambda_{q+1}^{\epsilon_{q+1}-1}n\right) e^{2\pi i \lambda_{q+1}^{\epsilon_{q+1}}  nk \cdot x}.
    \end{equation}
\end{lemma}
\begin{proof}
    Applying the Poisson summation formula to ~\eqref{eq:rho} gives ~\eqref{eq:fourier_series}.
\end{proof}

\section{Convex Integration Scheme}\label{Section3}
We present here a short outline of the argument to follow. We will consider the relaxed version of Equation ~\eqref{eq:stat_eqn} given by
\begin{equation}\label{eq:relaxed}
\begin{split}
    \nabla \cdot (\theta u) + \Lambda^\gamma \theta &= \nabla \cdot R\\
    u &= T\theta
    \end{split}
\end{equation}
for $R : \T^d \to \R^d$ which following \cite{IV} we refer to as the \textit{stress field}. We assume inductively that we have a smooth classical solution $(\theta_q,u_q,R_q)$ to ~\eqref{eq:relaxed}. If we put
\begin{equation}\label{eq:thetaq+1}
    \theta_{q+1} = \theta_q + w_{q+1}
\end{equation}
and
\begin{equation}\label{eq:uq+1}
    u_{q+1} = u_q + Tw_{q+1}
\end{equation}
for some carefully chosen smooth function $w_{q+1}:\T^d \to \R$, then $R_{q+1}$ must satisfy
\begin{equation*}
    \begin{split}
        \nabla \cdot R_{q+1} &= \nabla \cdot (\theta_{q+1}u_{q+1}) + \Lambda^\gamma \theta_{q+1}\\
        &= \nabla \cdot (R_q + w_{q+1}Tw_{q+1}) + \nabla \cdot (w_{q+1} T\theta_q + \theta_q Tw_{q+1}) + \Lambda^\gamma w_{q+1}
    \end{split}
\end{equation*}
Setting
\begin{equation}\label{eq:osc_error}
    R_O = R_q + w_{q+1}Tw_{q+1},
\end{equation}
\begin{equation}\label{eq:Nash_error}
    R_N = w_{q+1}T\theta_q + \theta_qTw_{q+1}
\end{equation}
and
\begin{equation}\label{eq:dis_error}
    R_D = -\Lambda^{\gamma-2} \nabla w_{q+1}
\end{equation}
we have that
\begin{equation}\label{eq:Rq+1}
    R_{q+1} = R_O + R_N + R_D + \nabla^\perp \varphi_{q+1}
\end{equation}
for some $\varphi_{q+1}: \T^d \to \R$; and in our case, we will show it is sufficient to take $\varphi_{q+1} = 0$. Equations ~\eqref{eq:osc_error} - ~\eqref{eq:dis_error} are referred to as the \textit{oscillation error}, \textit{Nash error}, and \textit{dissipation error}, respectively. We will show that for arbitrary $\alpha < 0$ and $1 \leq p < 2$, both $\Vert w_{q+1} \Vert_{\dot{B}^\alpha_{\infty,\infty}}$ and $\Vert w_{q+1} \Vert_{L^p}$ are summable, and so $\theta \in \dot{B}^\alpha_{\infty,\infty} \cap L^p$. Since $\alpha$ and $p$ are chosen arbitrarily, we get the desired regularity of $\theta$, and since $u$ is related to $\theta$ by a Calderon-Zygmund operator which is smoothing of degree $0$, we obtain the same regularity for $u$. Then upon showing that the dissipation, Nash, and oscillation errors all tend to $0$ in $\dot{H}^{-s}$ norm as $q \to \infty$, this will show that $u$ and $\theta$ are our desired solution to ~\eqref{eq:stat_eqn} in the sense of Definition \ref{def:weak_para_soln}. We note that the existence of $\mathbb{P}_{\not=0}(\theta u)$ as a paraproduct follows nearly immediately from our computations showing that the oscillation error can be made arbitrarily small.

\section{Inductive Proposition and Proof of Theorem \ref{thm:main}}\label{Section4}

\begin{proposition}\label{prop:ind}
    We make the following inductive assumptions about $(\theta_q,u_q,R_q)$:
    \begin{enumerate}
        \item\label{i:1} $\theta_q$ has zero mean and $u_q$ is divergence free;
        \item\label{i:2} $(\theta_q,u_q,R_q)$ is a smooth solution to the relaxed active scalar equation given by Equation ~\eqref{eq:relaxed};
        \item\label{i:3} $\Vert R_q \Vert_{\dot{H}^{-s}} < 2^{-q}$;
        \item\label{i:4} There is a constant $C(\alpha,p) > 0$ depending only on $\alpha < 0$ and $1 \leq p < 2$ such that for all $q' \leq q$ we have
        $$
        \Vert \theta_{q'} - \theta_{q'-1} \Vert_{\dot{B}^\alpha_{\infty,\infty}} + \Vert \theta_{q'} - \theta_{q'-1} \Vert_{L^p} \lesssim 2^{-C(\alpha,p)q'}
        $$
        where the implicit constant depends on $\alpha$ and $p$ but not $q$, or $q'$;
        \item\label{i:5} There is $\delta > 0$ independent of $q$ such that
        $$
        \Vert \theta_q \Vert_{L^1} > (1 + 2^{-q})\delta;
        $$
        \item\label{i:6} For each $q' \leq q$ there exists a unique $j \in \Z$ such that
        $$
        \mathbb{P}_{2^j}(\theta_{q'} - \theta_{q'-1}) = \theta_{q'} - \theta_{q'-1};
        $$
        \item\label{i:7} There are $C_1,C_2 > 0$ independent of $q$ such that
        $$
        \sum_{\substack{n,m \leq q\\ n \not  = m}} \Vert (\theta_n - \theta_{n-1}) (u_m - u_{m-1}) \Vert_{\dot{H}^{-s}} < C_1 - 2^{-q}
        $$
        and
        $$
         \sum_{n \leq q} \Vert (\theta_n - \theta_{n-1}) (u_n - u_{n-1}) \Vert_{\dot{H}^{-s}} < C_2 - 2^{-q + 100}
        $$
    \end{enumerate}
\end{proposition}

\begin{proof}[Proof of Theorem~\ref{thm:main} using Proposition~\ref{prop:ind}]
We start first by verifying the base case. For some $A > 0$ put $\theta_0 = A\cos(2\pi x_1)$, $u_0 = AT(\cos(2\pi x_1))$, and
$$
R_0 = \theta_0 u_0 + \Lambda^{\gamma - 2} \nabla \theta_0.
$$
Items \ref{i:1}, \ref{i:2}, and \ref{i:6} (taking $\theta_{-1} = 0$) are obvious. Taking $A > 0$ small enough ensures items \ref{i:3} and \ref{i:4}. Taking $\delta = A/2$ gives item \ref{i:5}. Choosing $C_1$ and $C_2$ large enough gives item \ref{i:7}.

Now assume all items hold for all $q$. Items \ref{i:1} and \ref{i:4} posit that $\{\theta_q\}$ is a Cauchy sequence in both $\dot{B}^\alpha_{\infty,\infty}(\T^d)$ and $L^p(\T^d)$ for all $\alpha < 0$ and all $p < 2$, and thus has a limit $\theta \in \dot{B}^\alpha_{\infty,\infty}(\T^d) \cap L^p(\T^d)$ for all $\alpha < 0$ and $p < 2$. This gives
$$
\theta \in \bigcap_{0 < \epsilon < 1} \dot{B}^{-\epsilon}_{\infty,\infty}(\T^d) \cap L^{2-\epsilon}(\T^d).
$$
From standard boundedness properties of Calderon-Zygmund operators we conclude $u$ lies in the same space as $\theta$.

From \ref{i:2}, we have that
$$
\int_{\T^d} \nabla \cdot (\theta_q u_q) \psi + \int_{\T^d} \psi \Lambda^\gamma \theta_q = \int_{\T^d} \psi \nabla \cdot R_q.
$$
Applying integration by parts and the self-adjointness of the fractional Laplacian gives
\begin{equation}\label{eq:integral_form}
    -\int_{\T^d} \mathbb{P}_{= 0}(\theta_q u_q) \cdot \nabla \psi + \int_{\T^d} \theta_q \Lambda^\gamma \psi = - \int_{\T^d} R_q \cdot \nabla \psi.
\end{equation}
We want to show that as $q \to \infty$ we obtain a solution in the sense of Definition \ref{def:weak_para_soln}. We start by analyzing the right hand side of ~\eqref{eq:integral_form}. We have
\begin{equation}\label{eq:stress_est}
    \left|\int_{\T^d} R_q \cdot \nabla \psi\right| = \left|\left\langle R_q, \nabla \psi\right\rangle_{\dot{H}^{-s},\dot{H}^s}\right| \lesssim \Vert R_q \Vert_{\dot{H}^{-s}} \Vert \nabla \psi \Vert_{\dot{H}^s} \lesssim 2^{-q} \Vert \nabla \psi \Vert_{\dot{H}^s} \to 0
\end{equation}
as we send $q \to \infty$. Now for the left hand side of Equation ~\eqref{eq:integral_form}, since $\theta_q \to \theta$ in $L^1$ we have
\begin{equation}\label{eq:L1_conv}
    \int_{\T^d} \theta_q \Lambda^\gamma \psi \to \int_{\T^d} \theta \Lambda^\gamma \psi = \langle \theta, \Lambda^\gamma \psi\rangle_{\dot{H}^{-\epsilon},\dot{H}^{\epsilon}}.
\end{equation}
Items \ref{i:6} and \ref{i:7} give that $\mathbb{P}_{\not=0}(\theta u)$ is well defined as a paraproduct in $\dot{H}^{-s}$. Hence
\begin{equation}\label{eq:nonlin_term}
    \begin{split}
        \int_{\T^d} \mathbb{P}_{= 0}(\theta_q u_q) \cdot \nabla \psi &= \int_{\T^d} \nabla \psi \cdot \sum_{\substack{n,m \leq q\\ j,j' \geq 0}} \mathbb{P}_{\not = 0}\left(\mathbb{P}_{2^j}(\theta_n - \theta_{n-1}) \mathbb{P}_{2^{j'}}(u_{m} - u_{m-1})\right)\\
        &= \left\langle \sum_{\substack{n,m \leq q\\ j,j' \geq 0}} \mathbb{P}_{\not = 0}\left(\mathbb{P}_{2^j}(\theta_n - \theta_{n-1}) \mathbb{P}_{2^{j'}}(u_{m} - u_{m-1})\right), \nabla \psi\right\rangle_{\dot{H}^{-s},\dot{H}^s}\\
        &\to \langle \mathbb{P}_{\not=0}(\theta u), \nabla \psi\rangle_{\dot{H}^{-s},\dot{H}^s}
    \end{split}
\end{equation}
Notice the third equality in ~\eqref{eq:nonlin_term} is defined this way using Definition \ref{def:paras}. Hence combining Equations ~\eqref{eq:stress_est}, ~\eqref{eq:L1_conv}, and ~\eqref{eq:nonlin_term} shows $\theta$ and $u$ is a weak solution to ~\eqref{eq:stat_eqn} in the sense of Definition \ref{def:weak_para_soln}; it remains to show that it is nontrivial. And for this, from item \ref{i:5} we see immediately that $\theta \not =0$. If $u = 0$, then we would have
$$
\int_{\T^d} \theta \Lambda^\gamma \psi = 0
$$
for all $\psi$. This forces $\theta = 0$ which is a contradiction. So $u \not =0$, and thus $\theta$ and $u$ are our desired solution.

\end{proof}

\section{Proof of Proposition~\ref{prop:ind}}\label{Section5}

\subsection{Parameters}\label{sec:param}

Here we specify the parameters that we will utilize throughout this section. First define the sequence $\{\epsilon_q\}_{q \in \N}$ by
$$
\epsilon_q = 1 - 2^{-q-1000}.
$$
Next we assume that $\{\lambda_q\}_{q \in \N}$ is a monotonically increasing sequence of very large powers of $2$ such that $\lambda_q^{\epsilon_q}$ is also an integer for all $q$. The size of $\lambda_{q+1}$ will be chosen to ensure that \eqref{eq:sigma}, \eqref{eq:diss_final_est}, \eqref{eq:R_N1_est}, \eqref{eq:R_N2_est}, \eqref{eq:Reynolds_cancel}, \eqref{eq:Bq+1k24_est}, \eqref{eq:bq+1k1_est}, \eqref{eq:high_freq_est_1}, \eqref{eq:diff_dir_est1}, \eqref{eq:L^p_est}, \eqref{eq:Besov_est_2}, \eqref{eq:delta_est}, \eqref{eq:para_est_1} and \eqref{eq:para_est_2} are all satisfied. Next we choose $0 < r \ll 1$ a power of $2$ such that
\begin{equation}\label{eq:r_constraint}
    r \leq \frac{5}{14}|k|^{-1},
\end{equation}
where $k \in \Omega$ from Lemma \ref{lem:lin_alg}, and choose $c$ such that
\begin{equation}\label{eq:c_constraint}
    |k|^{-1} + r \leq c \leq \frac{12}{7}|k|^{-1} - r
\end{equation}
and is of the form $c = (2a+1)/2^b$ for some $a,b \in \N$ such that ~\eqref{eq:c_constraint} is satisfied. Finally define $\{\sigma_q\}_{q \in \N}$ by
\begin{equation}\label{eq:sigma}
    \sigma_q = c\lambda_q \in \Z.
\end{equation}
Note we can ensure $\sigma_q \in \Z$ by requiring that $\lambda_q$ is a large enough power of $2$.

\subsection{Construction of $w_{q+1}$ and $L^p$ bounds}

\begin{definition}[\textbf{Potential Increment}]
    For $\rho_{q+1}^k$ from Lemma \ref{lem:boldW} put
    \begin{equation}\label{eq:wq+1}
        w_{q+1} = 2\sum_{k \in \Omega} \mathbb{P}_{\leq \lambda_{q+1}} \left(a_{k,q}(x) \rho^k_{q+1}(x)\right) \cos(2\pi \sigma_{q+1} k \cdot x)
    \end{equation}
    where for $\Gamma_k$ and $\epsilon_\Omega$ from Lemma \ref{lem:lin_alg} we put
    \begin{equation}
        a_{k,q} = \left(C_{q,k}'\frac{\Vert R_q\Vert_{L^\infty}}{\epsilon_\Omega}\right)^{-1/2}\Gamma_k\left(k^* - \epsilon_\Omega \frac{R_q}{\Vert R_q \Vert_{L^\infty}}\right).
    \end{equation}
    and
    $$
    C_{q,k}' = \int_{\R} \left|\hat{\phi}(x)\right|^2 \left|\hat{K}_{\simeq 1}(xk)\right|^2\, dx\footnote{Note that $C'_{q,k} \not = 0$ since $\phi$ is smooth and of compact support, and thus $\hat{\phi}$ is analytic, and hence $\hat{\phi}$ cannot vanish identically on the support of $\hat{K}_{\simeq 1}$. It is also clearly finite since this is the integral of a smooth function with compact support.}
    $$
    Also put
    \begin{equation}\label{eq:wq+1k}
        w_{q+1,k} = 2\mathbb{P}_{\leq \lambda_{q+1}} \left(a_{k,q}(x) \rho^k_{q+1}(x)\right) \cos(2\pi \sigma_{q+1} k \cdot x)
    \end{equation}
    so that
    $$
    w_{q+1} = \sum_{k \in \Omega} w_{q+1,k}
    $$
\end{definition}

\begin{lemma}[\textbf{$L^p$ bounds of $w_{q+1}$}]
    For $1 \leq p \leq \infty$ we have
    \begin{equation}\label{eq:wq+1k_est}
         \Vert w_{q+1,k} \Vert_{L^p} \lesssim \lambda_{q+1}^{(1-\epsilon_{q+1})\left(\frac{1}{2} - \frac{1}{p}\right)}
    \end{equation}
    and
    \begin{equation}\label{eq:wq+1_est}
        \Vert w_{q+1} \Vert_{L^p} \lesssim \lambda_{q+1}^{(1-\epsilon_{q+1})\left(\frac{1}{2} - \frac{1}{p}\right)}.
    \end{equation}
\end{lemma}
\begin{proof}
    We have using Lemmas \ref{lem:proj} and \ref{lem:boldW} that
    $$
    \Vert w_{q+1,k} \Vert_{L^p} \lesssim \Vert a_{k,q} \Vert_{L^\infty} \Vert \rho^k_{q+1}\Vert_{L^p} \lesssim \lambda_{q+1}^{(1-\epsilon_{q+1})\left(\frac{1}{2} - \frac{1}{p}\right)}
    $$
    which gives ~\eqref{eq:wq+1k_est}. ~\eqref{eq:wq+1_est} is immediate.
\end{proof}

\subsection{Note on Notation} As was already observed in the proof of equations ~\eqref{eq:wq+1k_est} and ~\eqref{eq:wq+1_est}, any constant that does not depend on $\lambda_{q+1}$ can be absorbed into the implicit constant present in all inequalities involving the $\lesssim$ symbol. Accordingly in what follows, all universal positive constants and all norms of $a_{k,q}$, $\theta_{q'}$, or $w_{q'}$ for $q' < q+1$ in the proceeding analysis will be treated in this manner.

\subsection{Proof of Item \ref{i:1}}
Suppose $\theta_q$ has zero mean. From ~\eqref{eq:r_constraint}, ~\eqref{eq:c_constraint}, and ~\eqref{eq:wq+1} it is clear that the frequency support of $w_{q+1}$ is located in an annulus with radius comparable to $\lambda_{q+1} \gg 0$. Hence $w_{q+1}$ has zero mean, and thus so does $\theta_{q+1} = \theta_q + w_{q+1}$. $u_{q+1}$ being divergence free is obvious.

\subsection{Proof of Item \ref{i:2}}
The fact $(\theta_{q+1},u_{q+1},R_{q+1})$ is a solution is obvious from our choices in ~\eqref{eq:thetaq+1}, ~\eqref{eq:uq+1}, and ~\eqref{eq:Rq+1} and our inductive assumption that $(\theta_q,u_q,R_q)$ is a solution. Smoothness of all involved terms follows from our inductive assumption as well as the smoothness of $w_{q+1}$ and the fact that $T$ applied to a smooth function is smooth.

\subsection{Proof of Item \ref{i:3}}\label{section:osc}
We assume that $\Vert R_q \Vert_{\dot{H}^{-s}} < 2^{-q}$, we aim to show that $\Vert R_{q+1} \Vert_{\dot{H}^{-s}} < 2^{-q-1}$. We will show this can be achieved by taking $\lambda_{q+1}$ sufficiently large.

\noindent\texttt{Dissipation error: } Using ~\eqref{eq:dis_error} we have
\begin{equation}\label{eq:dis_est}
    \Vert R_D \Vert_{\dot{H}^{-s}}^2 = \sum_{\xi \not =0} |\xi|^{-2s} |\xi|^{2\gamma-4} |\xi \hat{w}_{q+1}(\xi)|^2 \lesssim \Vert w_{q+1} \Vert_{L^1}^2 \sum_{\xi \not =0} |\xi|^{-2s+2\gamma-2}.
\end{equation}
Since we have chosen $s > \gamma + d/2- 1$, the sum in ~\eqref{eq:dis_est} converges, and hence applying ~\eqref{eq:wq+1_est} gives
\begin{equation}\label{eq:diss_final_est}
    \Vert R_D \Vert_{\dot{H}^{-s}} \lesssim \lambda_{q+1}^{\frac{\epsilon_{q+1}-1}{2}} < 2^{-2q-100}
\end{equation}
which holds for $\lambda_{q+1}$ chosen large enough.

\noindent\texttt{Nash error: } Put
\begin{equation*}
    R_{N,1} = w_{q+1}T\theta_q
\end{equation*}
and
\begin{equation*}
    R_{N,2} = \theta_q Tw_{q+1}.
\end{equation*}
From ~\eqref{eq:Nash_error} we have that $R_N = R_{N,1} + R_{N,2}$, and so we estimate each term separately. We have that
\begin{equation}\label{eq:R_N1_est}
    \Vert R_{N,1} \Vert_{\dot{H}^{-s}} \lesssim \Vert w_{q+1} T\theta_q \Vert_{L^1} \lesssim \Vert w_{q+1} \Vert_{L^{3/2}} \Vert T\theta_q \Vert_{L^3} \lesssim \lambda_{q+1}^{\frac{\epsilon_{q+1}-1}{6}} < 2^{-2q-101}
\end{equation}
for $\lambda_{q+1}$ chosen large enough and
\begin{equation}\label{eq:R_N2_est}
    \Vert R_{N,2} \Vert_{\dot{H}^{-s}} \lesssim \Vert \theta_q Tw_{q+1} \Vert_{L^1} \lesssim \Vert \theta_q \Vert_{L^3} \Vert w_{q+1} \Vert_{L^{3/2}} \lesssim \lambda_{q+1}^{\frac{\epsilon_{q+1}-1}{6}} < 2^{-2q-101}
\end{equation}
where we have invoked the $L^{3/2} \to L^{3/2}$ boundedness of $T$. Thus combining ~\eqref{eq:Nash_error}, ~\eqref{eq:R_N1_est}, and ~\eqref{eq:R_N2_est} gives
\begin{equation}\label{eq:Nash_final_est}
    \Vert R_N \Vert_{\dot{H}^{-s}} < (2)2^{-2q-101} = 2^{-2q-100}.
\end{equation}

\noindent\texttt{Oscillation error: } 

First let us write
\begin{equation}\label{eq:wTw_decomp_1}
    \begin{split}
        w_{q+1}Tw_{q+1} &= \left(\sum_{k \in \Omega} w_{q+1,k}\right)\left(\sum_{k\in \Omega} Tw_{q+1,k}\right)\\
        &= \sum_{k \in \Omega} w_{q+1,k}Tw_{q+1,k} + \sum_{k\not =k'} w_{q+1,k}Tw_{q+1,k'}.
    \end{split}
\end{equation}
The first summation on the second line contains terms that essentially have overlapping frequency support, and thus can have frequency support very near to the origin. This summation we refer to as the low frequency terms. Through geometric considerations, one can see the second summation has frequency support very far from the origin, and thus we refer to it as the high frequency component. The high frequency component will have very good estimates in $\dot{H}^{-s}$, whereas for the low frequency component this is where we expect to be able to cancel the stress field at level $q$. The remaining terms will then be of high frequency, and have good estimates in $\dot{H}^{-s}$.

We start with the low frequency component. Let us put
$$
w_{q+1,k}^+ = \mathbb{P}_{\leq \lambda_{q+1}} \left(a_{k,q}(x) \rho^k_{q+1}(x)\right) e^{2\pi i \sigma_{q+1} k \cdot x}
$$
and
$$
w_{q+1,k}^- = \mathbb{P}_{\leq \lambda_{q+1}} \left(a_{k,q}(x) \rho^k_{q+1}(x)\right) e^{-2\pi i \sigma_{q+1} k \cdot x}
$$
so that
$$
w_{q+1,k} = w_{q+1,k}^+ + w_{q+1,k}^-.
$$
So we have the following further decomposition of the low frequency component:
\begin{equation}\label{eq:wTw_decomp}
    w_{q+1,k} Tw_{q+1,k} = w_{q+1,k}^+Tw_{q+1,k}^+ + w_{q+1,k}^-Tw_{q+1,k}^+ + w_{q+1,k}^+Tw_{q+1,k}^- + w_{q+1,k}^-Tw_{q+1,k}^-.
\end{equation}
Since the frequency support of $w_{q+1,k}^+$ and $w_{q+1,k}^-$ does not become larger upon applying $T$, morally speaking we expect the first and last terms to be of high frequency, and thus have good decay estimates as we send $\lambda_{q+1} \to \infty$. However, the middle two terms we expect to have low frequency terms which we hope to utilize to cancel the stress field. This is the task we begin with. So we start by expanding
\begin{equation*}
    w_{q+1,k}^+ = \int_{\R^d} \hat{K}_{\simeq 1}\left(\frac{\xi - \sigma_{q+1}k}{\lambda_{q+1}}\right) \left(a_{k,q} \rho^k_{q+1}\right)^{\wedge}(\xi - \sigma_{q+1}k) e^{2\pi i \xi \cdot x}\, d\xi\footnote{Throughout this section as was done in \cite{BSV} and \cite{GR} we will shamelessly abuse notation and conflate the Fourier series and Fourier transform by identifying functions defined on $\T^d$ with their periodic extensions to $\R^d$. This will aid in the ease of the computations, but in practice there is no real harm; see \cite{CZ54} for a general transference principle}
\end{equation*}
and
\begin{equation*}
    Tw_{q+1,k}^- = \int_{\R^d} m(\eta)\hat{K}_{\simeq 1}\left(\frac{\eta + \sigma_{q+1}k}{\lambda_{q+1}}\right) \left(a_{k,q} \rho^k_{q+1}\right)^{\wedge}(\eta + \sigma_{q+1}k) e^{2\pi i \eta \cdot x}\, d\eta
\end{equation*}
to give
\begin{equation}\label{eq:bilin_exp^+-}
    w_{q+1,k}^+ Tw_{q+1,k}^- = \int_{\R^d} \int_{\R^d} M^{+,-}_{q+1,k}(\xi,\eta) \left(a_{k,q} \rho^k_{q+1}\right)^{\wedge}(\xi) \left(a_{k,q} \rho^k_{q+1}\right)^{\wedge}(\eta) e^{2\pi i (\xi + \eta) \cdot x}\, d\xi\, d\eta
\end{equation}
where
\begin{equation}\label{eq:M^+-}
     M_{q+1,k}^{+,-}(\xi,\eta) = m(\eta - \sigma_{q+1}k) \hat{K}_{\simeq 1}\left(\frac{\xi}{\lambda_{q+1}}\right) \hat{K}_{\simeq 1}\left(\frac{\eta}{\lambda_{q+1}}\right).
\end{equation}
Similarly to the procedure used to obtain ~\eqref{eq:bilin_exp^+-} and ~\eqref{eq:M^+-} we have 
\begin{equation}\label{eq:bilin_exp^-+}
    w_{q+1,k}^- Tw_{q+1,k}^+ = \int_{\R^d} \int_{\R^d} M^{-,+}_{q+1,k}(\xi,\eta) \left(a_{k,q} \rho^k_{q+1}\right)^{\wedge}(\xi) \left(a_{k,q} \rho^k_{q+1}\right)^{\wedge}(\eta) e^{2\pi i (\xi + \eta) \cdot x}\, d\xi\, d\eta
\end{equation}
and
\begin{equation}\label{eq:M^-+}
     M_{q+1,k}^{-,+}(\xi,\eta) = m(\eta + \sigma_{q+1}k) \hat{K}_{\simeq 1}\left(\frac{\xi}{\lambda_{q+1}}\right) \hat{K}_{\simeq 1}\left(\frac{\eta}{\lambda_{q+1}}\right).
\end{equation}
Thus combining ~\eqref{eq:bilin_exp^+-}, ~\eqref{eq:M^+-}, ~\eqref{eq:bilin_exp^-+}, and ~\eqref{eq:M^-+} we obtain that
\begin{equation}\label{eq:bilin_exp}
\begin{split}
    B_{q+1,k} &:= w_{q+1,k}^+ Tw_{q+1,k}^- + w_{q+1,k}^- Tw_{q+1,k}^+\\
    &= \int_{\R^d} \int_{\R^d} M_{q+1,k}(\xi,\eta) \left(a_{k,q} \rho^k_{q+1}\right)^{\wedge}(\xi) \left(a_{k,q} \rho^k_{q+1}\right)^{\wedge}(\eta) e^{2\pi i (\xi + \eta) \cdot x}\, d\xi\, d\eta
    \end{split}
\end{equation}
where
\begin{equation}\label{eq:Mq+1k}
    M_{q+1,k}(\xi,\eta) = \left[m\left(\eta - \sigma_{q+1}k\right) + m\left(\eta + \sigma_{q+1}k\right)\right] \hat{K}_{\simeq 1}\left(\frac{\xi}{\lambda_{q+1}}\right) \hat{K}_{\simeq 1}\left(\frac{\eta}{\lambda_{q+1}}\right).
\end{equation}
Put
\begin{equation}\label{eq:Mq+1^*}
    M^*_k(\xi,\eta) = \left[m\left(\eta - ck\right) + m\left(\eta + ck\right)\right] \hat{K}_{\simeq 1}\left(\xi\right) \hat{K}_{\simeq 1}\left(\eta\right).
\end{equation}
Clearly for $|\eta| < r$ we see using ~\eqref{eq:c_constraint} that
$$
|\eta \pm ck| \geq c|k| - |\eta| \geq |k|\left(|k|^{-1} + r\right) - r = 1 + \left(|k| - 1\right)r \geq 1.
$$
Thus $M_k^*$ is smooth since the support of $\hat{K}_{\simeq 1}$ forces $\eta - ck$ and $\eta + ck$ to avoid the singularity of the multiplier $m$\footnote{It is only at this point where we utilize the fact that $m$ is smooth away from the origin. Strictly speaking, this assumption can be weakened to merely assuming $m$ is smooth in a neighborhood of the vectors $\pm k$, $k \in \Omega$ and then taking $r$ small enough.\label{footnote}}. Now we note using the homogeneity of $m$ that
$$
M_{q+1,k}(\xi,\eta) = M_k^*\left(\frac{\xi}{\lambda_{q+1}},\frac{\eta}{\lambda_{q+1}}\right).
$$
Put
\begin{equation}\label{eq:kernel}
    K_{q+1,k}(y,z) = \lambda_{q+1}^{2d} \left(M_k^*\right)^{\vee}(\lambda_{q+1}y,\lambda_{q+1}z).
\end{equation}
Thus for multi-indices $\alpha,\beta$ we have that
\begin{equation}\label{eq:kernel_est}
    \Vert y^\alpha z^\beta K_{q+1,k}(y,z) \Vert_{L^1} \lesssim \lambda_{q+1}^{-|\alpha|-|\beta|}
\end{equation}
which follows from a change of variables. Now from ~\eqref{eq:bilin_exp} and ~\eqref{eq:kernel} we may write
\begin{equation}\label{eq:conv_exp}
    B_{q+1,k} = \int_{\R^d} \int_{\R^d} K_{q+1,k}(y,z) \rho_{q+1}^k(x-y) \rho_{q+1}^k(x-z) a_{k,q}(x-y) a_{k,q}(x-z)\, dy\, dz.
\end{equation}
To achieve the desired cancellation of the stress field, we decompose $\rho_{q+1}^k(x-y) \rho_{q+1}^k(x-z)$ into its mean and projection off the mean. To do this, utilizing Lemma \ref{lem:fourier} we first write
$$
\rho_{q+1}^k(x-y) \rho_{q+1}^k(x-z) = \sum_{n_1,n_2 \in \Z} \lambda_{q+1}^{\epsilon_{q+1}-1} \hat{\phi}\left(\lambda_{q+1}^{\epsilon_{q+1}-1}n_1\right) \hat{\phi}\left(\lambda_{q+1}^{\epsilon_{q+1}-1}n_2\right) e^{2\pi i \lambda_{q+1}^{\epsilon_{q+1}} \left(n_1k \cdot (x-y) + n_2k \cdot (x-z)\right)}.
$$
Clearly then
\begin{equation}\label{eq:mean_prelim}
    \mathbb{P}_{=0,x}\left(\rho_{q+1}^k(x-y) \rho_{q+1}^k(x-z)\right) = - \sum_{n \in \Z} \lambda_{q+1}^{\epsilon_{q+1}-1} \left(\hat{\phi}\left(\lambda_{q+1}^{\epsilon_{q+1}-1}n\right)\right)^2 e^{-2\pi i \lambda_{q+1}^{\epsilon_{q+1}} n k \cdot (y-z)}
\end{equation}
and
\begin{equation}\label{eq:off_mean}
\begin{split}
    \mathbb{P}_{\not=0,x}\left(\rho_{q+1}^k(x-y) \rho_{q+1}^k(x-z)\right) &= \sum_{n_1+n_2 \not = 0} \lambda_{q+1}^{\epsilon_{q+1}-1} \hat{\phi}\left(\lambda_{q+1}^{\epsilon_{q+1}-1}n_1\right) \hat{\phi}\left(\lambda_{q+1}^{\epsilon_{q+1}-1}n_2\right)\\
    &\times e^{2\pi i \lambda_{q+1}^{\epsilon_{q+1}} \left(n_1k \cdot (x-y) + n_2k \cdot (x-z)\right)}.
    \end{split}
\end{equation}
We use the $x$ subscript in ~\eqref{eq:mean_prelim} and ~\eqref{eq:off_mean} to indicate the mean is with respect to the $x$ variable. Note in ~\eqref{eq:mean_prelim}, since we assume that $\phi$ is odd, $\hat{\phi}$ is purely imaginary. Hence $-(\hat{\phi}(\xi))^2 = |\hat{\phi}(\xi)|^2$ and so we can rewrite ~\eqref{eq:mean_prelim} as
\begin{equation}\label{eq:mean}
    \mathbb{P}_{=0,x}\left(\rho_{q+1}^k(x-y) \rho_{q+1}^k(x-z)\right) = \sum_{n \in \Z} \lambda_{q+1}^{\epsilon_{q+1}-1} \left|\hat{\phi}\left(\lambda_{q+1}^{\epsilon_{q+1}-1}n\right)\right|^2 e^{-2\pi i \lambda_{q+1}^{\epsilon_{q+1}} n k \cdot (y-z)}.
\end{equation}
Now using this decomposition we can rewrite ~\eqref{eq:conv_exp} as
\begin{equation}\label{eq:conv_exp_decomp}
    \begin{split}
        B_{q+1,k} &= \int_{\R^d} \int_{\R^d} K_{q+1,k}(y,z) \mathbb{P}_{\not=0,x}\left(\rho_{q+1}^k(x-y) \rho_{q+1}^k(x-z)\right) a_{k,q}(x-y) a_{k,q}(x-z)\, dy\, dz\\
        &+ \int_{\R^d} \int_{\R^d} K_{q+1,k}(y,z) \mathbb{P}_{=0,x}\left(\rho_{q+1}^k(x-y) \rho_{q+1}^k(x-z)\right) a_{k,q}(x-y) a_{k,q}(x-z)\, dy\, dz
    \end{split}
\end{equation}
We denote the expression on the top line of ~\eqref{eq:conv_exp_decomp} as $B_{q+1,k}^1$ and the expression on the second line as $B_{q+1,k}^2$. Starting with $B_{q+1,k}^2$, we have utilizing the fundamental theorem of calculus
\begin{equation*}\label{eq:FTC1}
    a_{k,q}(x-y) = a_{k,q}(x) - y \cdot \int_0^1 \nabla a_{k,q}(x - ty)\, dt.
\end{equation*}
and
\begin{equation*}
    a_{k,q}(x-z) = a_{k,q}(x) - z \cdot \int_0^1 \nabla a_{k,q}(x - tz)\, dt.
\end{equation*}
so
\begin{equation*}
    \begin{split}
        a_{k,q}(x-y)a_{k,q}(x-z) &= a_{k,q}(x)^2\\
        &- a_{k,q}(x) \left(y \cdot \int_0^1 \nabla a_{k,q}(x - ty)\, dt + z \cdot \int_0^1 \nabla a_{k,q}(x - tz)\, dt\right)\\
        &+ \left(y \cdot \int_0^1 \nabla a_{k,q}(x - ty)\, dt\right)\left(z \cdot \int_0^1 \nabla a_{k,q}(x - tz)\, dt\right)
    \end{split}
\end{equation*}
Utilizing this, we may rewrite $B_{q+1,k}^2$ as
\begin{equation}\label{eq:Bq+1k2_decomp}
    \begin{split}
        B_{q+1,k}^2 &= a_{k,q}^2 \int_{\R^d} \int_{\R^d} K_{q+1,k}(y,z) \mathbb{P}_{\not=0,x}\left(\rho_{q+1}^k(x-y) \rho_{q+1}^k(x-z)\right)\, dy\, dz\\
        &- a_{k,q} \int_{\R^d} \int_{\R^d} \int_0^1 K_{q+1,k}(y,z) \mathbb{P}_{\not=0,x}\left(\rho_{q+1}^k(x-y) \rho_{q+1}^k(x-z)\right) y \cdot \nabla a_{k,q}(x - ty) \, dt\, dy\, dz\\
        &- a_{k,q} \int_{\R^d} \int_{\R^d} \int_0^1 K_{q+1,k}(y,z) \mathbb{P}_{\not=0,x}\left(\rho_{q+1}^k(x-y) \rho_{q+1}^k(x-z)\right) z \cdot \nabla a_{k,q}(x - tz) \, dt\, dy\, dz\\
        &+ \int_{\R^d} \int_{\R^d} \int_{[0,1]^2} K_{q+1,k}(y,z) \mathbb{P}_{\not=0,x}\left(\rho_{q+1}^k(x-y) \rho_{q+1}^k(x-z)\right)\\
        &\times y\cdot \nabla a_{k,q}(x-t_1y) z \cdot \nabla a_{k,q}(x - t_2z) \, dt_1\, dt_2\, dy\, dz.
    \end{split}
\end{equation}
We denote each line on the right hand side of ~\eqref{eq:Bq+1k2_decomp} by $B_{q+1,k}^{2,1}$, $B_{q+1,k}^{2,2}$, $B_{q+1,k}^{2,3}$, and $B_{q+1,k}^{2,4}$ respectively. First considering $B_{q+1,k}^{2,1}$ and applying the homogeneity of $m$ we have that
\begin{equation}\label{eq:Bq+1k2,2_1}
    \begin{split}
        B_{q+1,k}^{2,1} &= a_{k,q}^2 \sum_{n \in \Z} \left|\hat{\phi}\left(\lambda_{q+1}^{\epsilon_{q+1}-1}n\right)\right|^2 \int_{\R^d} \int_{\R^d} \lambda_{q+1}^{\epsilon_{q+1}-1} K_{q+1,k}(y,z) e^{-2\pi i \lambda_{q+1}^{\epsilon_{q+1}} nk \cdot (y-z)}\, dy\, dz\\
        &= a_{k,q}^2 \sum_{n \in \Z} \lambda_{q+1}^{\epsilon_{q+1}-1} \left|\hat{\phi}\left(\lambda_{q+1}^{\epsilon_{q+1}-1}n\right)\right|^2 \hat{K}_{q+1,k}\left(\lambda_{q+1}^{\epsilon_{q+1}} nk, -\lambda_{q+1}^{\epsilon_{q+1}} nk\right)\\
        &= a_{k,q}^2 \sum_{n \in \Z} \lambda_{q+1}^{\epsilon_{q+1}-1} \left|\hat{\phi}\left(\lambda_{q+1}^{\epsilon_{q+1}-1}n\right)\right|^2 [m(-\lambda_{q+1}^{\epsilon_{q+1}} nk -\sigma_{q+1}k)\\
        &+ m(-\lambda_{q+1}^{\epsilon_{q+1}} nk +\sigma_{q+1}k)]\left|\hat{K}_{\simeq 1}\left(\lambda_{q+1}^{\epsilon_{q+1}-1}nk\right)\right|^2\\
        &= a_{k,q}^2 \left[m(-k) + m(k)\right] \sum_{n \in \Z} \lambda_{q+1}^{\epsilon_{q+1}-1} \left|\hat{\phi}\left(\lambda_{q+1}^{\epsilon_{q+1}-1}n\right)\right|^2 \left|\hat{K}_{\simeq 1}\left(\lambda_{q+1}^{\epsilon_{q+1}-1}nk\right)\right|^2
    \end{split}
\end{equation}
From standard results in Riemann integration theory we have that
\begin{equation}\label{eq:Riemann}
    \int_{\R} |\hat{\phi}(x)|^2 \left|\hat{K}_{\simeq 1}\left(xk\right)\right|^2\, dx = \sum_{n \in \Z} \lambda_{q+1}^{\epsilon_{q+1}-1} \left|\hat{\phi}\left(\lambda_{q+1}^{\epsilon_{q+1}-1}n\right)\right|^2 \left|\hat{K}_{\simeq 1}\left(\lambda_{q+1}^{\epsilon_{q+1}-1}nk\right)\right|^2 + O\left(\lambda_{q+1}^{\epsilon_{q+1}-1}\right).
\end{equation}
So if we put
$$
C_{q+1,k} = \int_{\R} |\hat{\phi}(x)|^2 \left|\hat{K}_{\simeq 1}\left(xk\right)\right|^2\, dx - \sum_{n \in \Z} \lambda_{q+1}^{\epsilon_{q+1}-1} \left|\hat{\phi}\left(\lambda_{q+1}^{\epsilon_{q+1}-1}n\right)\right|^2 \left|\hat{K}_{\simeq 1}\left(\lambda_{q+1}^{\epsilon_{q+1}-1}nk\right)\right|^2
$$
then from ~\eqref{eq:Riemann} we have that $|C_{q+1,k}| \lesssim \lambda_{q+1}^{\epsilon_{q+1}-1}$. So combining ~\eqref{eq:Bq+1k2,2_1} and ~\eqref{eq:Riemann} we obtain
\begin{equation*}
    B_{q+1,k}^{2,1} = \left[m(k) + m(-k)\right]\left(a_{k,q}^2 \int_{\R} \left|\hat{\phi}(x)\right|^2 \left|\hat{K}_{\simeq 1}(xk)\right|^2\, dx - C_{q+1,k} a_{k,q}^2\right).
\end{equation*}
Summing over $k \in \Omega$ and applying Lemma \ref{lem:lin_alg} we have
\begin{equation*}
    \sum_{k \in \Omega} B_{q+1,k}^{2,1} = \frac{\epsilon_\Omega}{\Vert R_q \Vert_{L^\infty}} k^* - R_q + \sum_{k \in \Omega} C_{q+1,k}a_{k,q}^2 \left[m(k) + m(-k)\right]
\end{equation*}
and thus
\begin{equation}\label{eq:Reynolds_cancel}
\begin{split}
    \left\Vert R_q + \sum_{k \in \Omega} B_{q+1,k}^{2,1} \right\Vert_{\dot{H}^{-s}} &= \left\Vert \frac{\epsilon_\Omega}{\Vert R_q \Vert_{L^\infty}}k^* + \sum_{k \in \Omega} C_{q+1,k} a_{k,q}^2 \left[m(k) + m(-k)\right]\right\Vert_{\dot{H}^{-s}}\\
    &\lesssim \sum_{k \in \Omega} |C_{q+1,k}| \Vert a_{k,q} \Vert_{L^\infty}^2\\
    &\lesssim \lambda_{q+1}^{\epsilon_{q+1}-1}\\
    &< 2^{-2q-100}
\end{split}
\end{equation}
for $\lambda_{q+1}$ chosen large enough. 

Note estimating $B_{q+1,k}^{2,2}$ and $B_{q+1,k}^{2,3}$ is the same procedure, so we choose to focus on $B_{q+1,k}^{2,2}$. From ~\eqref{eq:mean} that we have
\begin{equation}\label{eq:mean_est}
    \left| \mathbb{P}_{=0,x}\left(\rho_{q+1}^k(x-y) \rho_{q+1}^k(x-z)\right)\right| \lesssim \sum_{n \in \Z} \lambda_{q+1}^{\epsilon_{q+1}-1} \left|\hat{\phi}\left(\lambda_{q+1}^{\epsilon_{q+1}-1}n\right)\right|^2 \simeq \Vert \phi \Vert_{L^2}^2 \simeq 1.
\end{equation}
So from ~\eqref{eq:kernel_est} and ~\eqref{eq:mean_est} we obtain
\begin{equation}\label{eq:Bq+1k22_est}
    \begin{split}
        \left\Vert B_{q+1,k}^{2,2} \right\Vert_{\dot{H}^{-s}} &\lesssim \int_{\R^d} \int_{\R^d} \int_0^1 \left|y K(y,z)\right| \left| \mathbb{P}_{=0,x}\left(\rho_{q+1}^k(x-y) \rho_{q+1}^k(x-z)\right)\right| \left\Vert a_{k,q} \nabla a_{k,q}\right\Vert_{\dot{H}^{-s}}\, dt\, dy\, dz\\
        &\lesssim \lambda_{q+1}^{-1}\\
        &< d^{-1}2^{-2q-100}.
    \end{split}
\end{equation}
Utilizing the same methodology we obtain
\begin{equation}\label{eq:Bq+1k23_est}
    \left\Vert B_{q+1,k}^{2,3} \right\Vert_{\dot{H}^{-s}} < d^{-1} 2^{-2q-100}.
\end{equation}
Finally, note that for multi-indices $|\alpha| + |\beta| = 2$ we have that a general term of $B_{q+1,k}^{2,4}$ is given by
\begin{equation}\label{eq:Bq+1k22_decomp}
    \begin{split}
        B_{q+1,k}^{2,4,\alpha,\beta} &:= \int_{\R^d} \int_{\R^d} \int_{[0,1]^2} y^\alpha z^\beta K(y,z) \mathbb{P}_{=0,x}\left(\rho_{q+1}^k(x-y) \rho_{q+1}^k(x-z)\right)\\
        &\times \nabla^\alpha a_{k,q}(x-t_1y) \nabla^\beta a_{k,q}(x - t_2z)\, dt_1\, dt_2\, dy\, dz.
    \end{split}
\end{equation}
and clearly
$$
B_{q+1,k}^{2,4} = \sum_{|\alpha| + |\beta| = 2} B_{q+1,k}^{2,4,\alpha,\beta}.
$$
Utilizing ~\eqref{eq:kernel_est} and ~\eqref{eq:mean_est}, and proceeding much in the same way as ~\eqref{eq:Bq+1k22_est} we obtain
\begin{equation}\label{eq:Bq+1k24_est}
    \begin{split}
        \left\Vert B^{2,4}_{q+1,k}\right\Vert_{\dot{H}^{-s}} &\lesssim \sum_{|\alpha| + |\beta| = 2} \left\Vert B^{2,4,\alpha,\beta}_{q+1,k}\right\Vert_{\dot{H}^{-s}}\\
        &\lesssim \lambda_{q+1}^{-2}\\
        &< d^{-1} 2^{-2q-100}
    \end{split}
\end{equation}
for $\lambda_{q+1}$ chosen large enough. Combining ~\eqref{eq:Reynolds_cancel}, ~\eqref{eq:Bq+1k22_est}, ~\eqref{eq:Bq+1k23_est}, ~\eqref{eq:Bq+1k24_est}, and the fact $|\Omega| = d$ gives
\begin{equation}\label{eq:Bq+1k2_est}
    \left\Vert R_q + \sum_{k \in \Omega} B_{q+1,k}^2\right\Vert_{\dot{H}^{-s}} < 2^{-2q-100} + (3)(d)\left(d^{-1}2^{-2q-100}\right) < 2^{-2q-97}.
\end{equation}

It remains to treat $B_{q+1,k}^1$. For this, utilizing \cite[Lemma 2.16]{GR}\footnote{The proof crucially utilizes \cite[Proposition 1]{BOZ}} we have that
\begin{equation}\label{eq:kato_ponce_est}
    \left\Vert \mathbb{P}_{\not=0,x}\left(\rho_{q+1}^k(x-y) \rho_{q+1}^k(x-z)\right) a_{k,q}(x-y) a_{k,q}(x-z) \right\Vert_{\dot{H}^{-s}_x} \lesssim \left\Vert \rho_{q+1}^k(x-y) \rho_{q+1}^k(x-z)\right\Vert_{\dot{H}^{-s}_x}.
\end{equation}
As before we utilize the $x$ subscript to emphasize the norms are with respect to the $x$ variable. Putting
$$
\mathcal{E}_{y,z}(n_1,n_2) = e^{-2\pi i \lambda_{q+1}^{\epsilon_{q+1}} k \cdot (n_1y + n_2z)}
$$
from ~\eqref{eq:off_mean} we see for $\xi \not =0$ that
\begin{equation}
    \left(\rho_{q+1}^k(x-y) \rho_{q+1}^k(x-z)\right)^{\wedge}(\xi) = \sum_{(n_1,n_2) \in S(\xi)} \lambda_{q+1}^{\epsilon_{q+1}-1} \hat{\phi}\left(\lambda_{q+1}^{\epsilon_{q+1}-1}n_1\right) \hat{\phi}\left(\lambda_{q+1}^{\epsilon_{q+1}-1}n_2\right) \mathcal{E}_{y,z}(n_1,n_2)
\end{equation}
where
$$
S(\xi) = \{(n_1,n_2) \in \Z^d : \lambda_{q+1}^{\epsilon_{q+1}} (n_1 + n_2)k = \xi\}.
$$
If $S(\xi) = \emptyset$ then the Fourier coefficient is of course $0$. So now we compute
\begin{equation}\label{eq:mean_est_2}
    \begin{split}
        &\left\Vert \rho_{q+1}^k(x-y) \rho_{q+1}^k(x-z)\right\Vert_{\dot{H}^{-s}_x}^2\\
        =& \sum_{\xi \not = 0} |\xi|^{-2s} \left|\sum_{(n_1,n_2) \in S(\xi)} \lambda_{q+1}^{\epsilon_{q+1}-1} \hat{\phi}\left(\lambda_{q+1}^{\epsilon_{q+1}-1}n_1\right) \hat{\phi}\left(\lambda_{q+1}^{\epsilon_{q+1}-1}n_2\right) \mathcal{E}_{y,z}(n_1,n_2)\right|^2\\
        \lesssim& \lambda_{q+1}^{2(\epsilon_{q+1}-1)} \sum_{j \in \Z \setminus \{0\}} |\lambda_{q+1}^{\epsilon_{q+1}} jk|^{-2s} \left(\sum_{n_1+n_2 = j} \left|\hat{\phi}\left(\lambda_{q+1}^{\epsilon_{q+1}-1}n_1\right)\right| \left|\hat{\phi}\left(\lambda_{q+1}^{\epsilon_{q+1}-1}n_2\right)\right|\right)^2\\
        \simeq& \lambda_{q+1}^{2\epsilon_{q+1}(1-s)-2} \sum_{j \in \Z \setminus \{0\}} |j|^{-2s} \left| \left( \left|\hat{\phi}\left(\lambda_{q+1}^{\epsilon_{q+1}-1}\cdot\right)\right| \ast \left|\hat{\phi}\left(\lambda_{q+1}^{\epsilon_{q+1}-1}\cdot\right)\right|\right)(j)\right|^2.
    \end{split}
\end{equation}
To complete the estimate, we utilize Young's convolution inequality to compute
\begin{equation}\label{eq:young}
    \begin{split}
        \left\Vert \left|\hat{\phi}\left(\lambda_{q+1}^{\epsilon_{q+1}-1}\cdot\right)\right| \ast \left|\hat{\phi}\left(\lambda_{q+1}^{\epsilon_{q+1}-1}\cdot\right)\right| \right\Vert_{\ell^\infty}^2 &\lesssim \left\Vert \hat{\phi}\left(\lambda_{q+1}^{\epsilon_{q+1}-1} \cdot\right)\right\Vert_{\ell^2}^4\\
        &= \left(\sum_{n \in \Z} \left|\hat{\phi}\left(\lambda_{q+1}^{\epsilon_{q+1}-1} n\right)\right|^2\right)^2\\
        &= \lambda_{q+1}^{2(1-\epsilon_{q+1})} \left(\sum_{n \in \Z} \lambda_{q+1}^{\epsilon_{q+1}-1} \left|\hat{\phi}\left(\lambda_{q+1}^{\epsilon_{q+1}-1} n\right)\right|^2\right)^2\\
        &\simeq \lambda_{q+1}^{2(1-\epsilon_{q+1})} \Vert \phi \Vert_{L^2}^4\\
        &\simeq \lambda_{q+1}^{2(1-\epsilon_{q+1})}
    \end{split}
\end{equation}
and hence combining ~\eqref{eq:mean_est_2} and ~\eqref{eq:young} gives
\begin{equation}\label{eq:off_mean_est_final}
    \left\Vert \rho_{q+1}^k(x-y) \rho_{q+1}^k(x-z)\right\Vert_{\dot{H}^{-s}_x} \lesssim \lambda_{q+1}^{-\epsilon_{q+1} s}.
\end{equation}
And so from ~\eqref{eq:kernel_est}, ~\eqref{eq:kato_ponce_est}, and ~\eqref{eq:off_mean_est_final} we deduce
\begin{equation}\label{eq:bq+1k1_est}
\begin{split}
    \left\Vert B_{q+1,k}^1\right\Vert_{\dot{H}^{-s}} &= \left\Vert \int_{\R^d} \int_{\R^d} K(y,z) \mathbb{P}_{\not=0,x}\left(\rho_{q+1}^k(x-y) \rho_{q+1}^k(x-z)\right) a_{k,q}(x-y)a_{k,q}(x-z)\, dy\, dz\right\Vert_{\dot{H}^{-s}_x}\\
    &\lesssim \int_{\R^d} \int_{\R^d} |K(y,z)| \left\Vert \mathbb{P}_{\not=0,x}\left(\rho_{q+1}^k(x-y) \rho_{q+1}^k(x-z)\right) \right\Vert_{\dot{H}_x^{-s}}\, dy\, dz\\
    &\lesssim \lambda_{q+1}^{-\epsilon_{q+1} s}\\
    &< d^{-1}2^{-2q-100}
    \end{split}
\end{equation}
for $\lambda_{q+1}$ chosen large enough. And so from ~\eqref{eq:Bq+1k2_est} and ~\eqref{eq:bq+1k1_est} this gives
\begin{equation}\label{eq:low_freq_est}
    \begin{split}
        \left\Vert  R_q + \sum_{k \in \Omega} \left(w_{q+1,k}^+ Tw_{q+1,k}^- + w_{q+1,k}^- Tw_{q+1,k}^+\right) \right\Vert_{\dot{H}^{-s}} &\leq \left\Vert R_q + \sum_{k \in \Omega} B^2_{q+1,k} \right\Vert_{\dot{H}^{-s}} + \sum_{k \in \Omega} \left\Vert B^1_{q+1,k}\right\Vert_{\dot{H}^{-s}}\\
        &< 2^{-2q-97} + d\left(d^{-1}2^{-2q-100}\right)\\
        &< 2^{-2q-96}.
    \end{split}
\end{equation}
This completes the analysis of the two low frequency terms identified in ~\eqref{eq:wTw_decomp}. The remaining two high frequency terms are treated in exactly the same way as each other, so we focus on $w_{q+1}^+ Tw_{q+1}^+$. Following the procedure to obtain ~\eqref{eq:bilin_exp^+-}, we have
\begin{equation}
    w_{q+1,k}^+ Tw_{q+1,k}^+ = e^{2\pi i \sigma_{q+1}k \cdot x}\int_{\R^d} \int_{\R^d} M^{+,+}_{q+1,k}(\xi,\eta) \left(a_{k,q} \rho^k_{q+1}\right)^{\wedge}(\xi) \left(a_{k,q} \rho^k_{q+1}\right)^{\wedge}(\eta) e^{2\pi i (\xi + \eta) \cdot x}\, d\xi\, d\eta
\end{equation}
where in this instance $M_{q+1,k}^{+,+} = M_{q+1,k}^{-,+}$. Importantly, this means the inverse Fourier transform of $M_{q+1,k}^{+,+}$ obeys the same estimates given in ~\eqref{eq:kernel_est}. For us, this means
$$
\left\Vert M_{q+1,k}^{+,+} \right\Vert_{L^\infty} \lesssim 1.
$$
We also note that $M_{q+1,k}^{+,+}(\xi,\eta)$ is supported in $B(0,r\lambda_{q+1}) \times B(0,r\lambda_{q+1})$. Now we take the Fourier transform of $w_{q+1,k}^+ Tw_{q+1,k}^+$ to get
\begin{equation*}
    \left(w_{q+1,k}^+Tw_{q+1,k}^+\right)^\wedge(\zeta) = \int_{\R^6} M_{q+1,k}^{+,+}(\xi,\eta) \left(a_{k,q} \rho^k_{q+1}\right)^{\wedge}(\xi) \left(a_{k,q} \rho^k_{q+1}\right)^{\wedge}(\eta) e^{2\pi i (\xi + \eta -\zeta + \sigma_{q+1}k) \cdot x}\, d\xi\, d\eta\, dx
\end{equation*}
In the sense of distributions, we have that
$$
\int_{\R^d} e^{2\pi i (\xi + \eta -\zeta + \sigma_{q+1}k)}\, dx = \delta(\xi + \eta - \zeta + \sigma_{q+1}k)
$$
so
$$
\left(w_{q+1,k}^+Tw_{q+1,k}^+\right)^\wedge(\zeta) = \int_{\R^d} M_{q+1,k}^{+,+}(\xi,\zeta - \xi - \sigma_{q+1}k) \left(a_{k,q} \rho^k_{q+1}\right)^{\wedge}(\xi) \left(a_{k,q} \rho^k_{q+1}\right)^{\wedge}(\zeta - \xi - \sigma_{q+1}k)\, d\xi.
$$
From the support of $M_{q+1,k}^{+,+}$, we see that
$$
|\xi| \leq r\lambda_{q+1} \quad \text{and} \quad |\zeta - \xi - \sigma_{q+1}k| \leq r\lambda_{q+1}
$$
thus
\begin{equation*}
    |\zeta - \sigma_{q+1}k| \leq |\xi| + |\zeta - \xi - \sigma_{q+1}k| \leq 2r\lambda_{q+1}.
\end{equation*}
Thus $\left(w_{q+1,k}^+Tw_{q+1,k}^+\right)^\wedge$ is supported in $B(\sigma_{q+1}k,2r\lambda_{q+1})$ which we see is contained in the annulus $|\zeta| \simeq \lambda_{q+1}$ from ~\eqref{eq:r_constraint} and ~\eqref{eq:c_constraint}. So then
\begin{equation*}
    \begin{split}
        \left\Vert w_{q+1,k}^+Tw_{q+1,k}^+ \right\Vert_{\dot{H}^{-s}}^2 &= \sum_{\zeta \not =0} |\zeta|^{-2s} \left|\left(w_{q+1,k}^+Tw_{q+1,k}^+\right)^\wedge(\zeta)\right|^2\\
        &= \sum_{\zeta \in B(\sigma_{q+1}k,2r\lambda_{q+1})} |\zeta|^{-2s} \left|\left(w_{q+1,k}^+Tw_{q+1,k}^+\right)^\wedge(\zeta)\right|^2\\
        &\lesssim \left(\lambda_{q+1}^{-2s}\right) \lambda_{q+1}^d \left\Vert a_{k,j} \rho_{q+1}^k \right\Vert_{L^2}^4\\
        &\lesssim \lambda_{q+1}^{2(d/2-s)}
    \end{split}
\end{equation*}
and thus since $s$ was chosen to be larger than $d/2$ we have
\begin{equation}\label{eq:high_freq_est_1}
    \left\Vert w_{q+1,k}^+Tw_{q+1,k}^+ \right\Vert_{\dot{H}^{-s}} \lesssim \lambda_{q+1}^{d/2-s} < d^{-1}2^{-2q-100}
\end{equation}
for $\lambda_{q+1}$ chosen large enough. In exactly the same manner, one can establish that
\begin{equation*}
    \left\Vert w_{q+1,k}^-Tw_{q+1,k}^- \right\Vert_{\dot{H}^{-s}} < d^{-1}2^{-2q-100}
\end{equation*}
and so
\begin{equation}\label{eq:high_freq_est_2}
    \left\Vert \sum_{k \in \Omega} w_{q+1,k}^+Tw_{q+1,k}^+ + w_{q+1,k}^-Tw_{q+1,k}^- \right\Vert_{\dot{H}^{-s}} < 2^{-2q-98}.
\end{equation}
Finally from ~\eqref{eq:low_freq_est} and ~\eqref{eq:high_freq_est_2} we have
\begin{equation}\label{eq:low_freq_est_final}
    \left\Vert R_q + \sum_{k \in \Omega} w_{q+1,k}Tw_{q+1,k}\right\Vert_{\dot{H}^{-s}} < 2^{-2q-96} + 2^{-2q-98} < 2^{-2q-95}.
\end{equation}
This concludes the analysis of the low frequency component identified in ~\eqref{eq:wTw_decomp_1}. For the high frequency component, fix $k,k' \in \Omega$ with $k \not= k'$. We have that
$$
\Vert w_{q+1,k}^+Tw_{q+1,k'}^+\Vert_{\dot{H}^{-s}}^2 = \sum_{\xi \not = 0} |\xi|^{-2s} \left|\sum_{\eta \in \Z^d} \hat{w}_{q+1,k}^+(\eta) m(\xi - \eta) \hat{w}_{q+1,k'}^+(\xi - \eta)\right|^2.
$$
In order for $\hat{w}_{q+1,k}^+(\eta) \hat{w}^+_{q+1,k'}(\xi - \eta) \not  = 0$, we require that
$$
\eta \in B(\sigma_{q+1,k}k, r\lambda_{q+1}) \quad \text{and} \quad \xi - \eta \in B(\sigma_{q+1}k',r\lambda_{q+1})
$$
and so
$$
|\eta - \sigma_{q+1}k| \leq r\lambda_{q+1} \quad \text{and} \quad |\xi - \eta - \sigma_{q+1}k'| \leq r\lambda_{q+1}.
$$
Thus
$$
|\xi - \sigma_{q+1}(k+k')| \leq 2r\lambda_{q+1}
$$
and so
$$
\xi \in B(\sigma_{q+1}(k+k'), 2r\lambda_{q+1}).
$$
Note that
$$
|\xi| \geq |\sigma_{q+1} - 2r\lambda_{q+1}||k+k'| = (c - 2r)\lambda_{q+1}|k+k'| \geq \frac{9|k+k'|}{14|k|} \lambda_{q+1} \simeq \lambda_{q+1}
$$
where we have utilized ~\eqref{eq:r_constraint} and ~\eqref{eq:c_constraint} as well as the fact that $\Omega$ is a basis so $k + k' \not= 0$. Therefore we have
\begin{equation*}
    \begin{split}
        \Vert w_{q+1,k}^+Tw_{q+1,k'}^+\Vert_{\dot{H}^{-s}}^2 &= \sum_{|\xi| \simeq \lambda_{q+1}} |\xi|^{-2s} \left|\left(w_{q+1,k}^+ Tw_{q+1,k'}^+\right)^{\wedge}(\xi)\right|^2\\
        &\lesssim \lambda_{q+1}^{2(d/2-s)} \Vert w_{q+1,k}^+ Tw_{q+1,k'}^+ \Vert_{L^1}^2\\
        &\lesssim \lambda_{q+1}^{2(d/2-s)} \Vert w_{q+1,k}^+ \Vert_{L^2}^4\\
        &\lesssim \lambda_{q+1}^{2(d/2-s)}.
    \end{split}
\end{equation*}
Hence for $\lambda_{q+1}$ large enough we obtain
\begin{equation}\label{eq:diff_dir_est1}
    \Vert w_{q+1,k}^+Tw_{q+1,k'}^+\Vert_{\dot{H}^{-s}} < d^{-1}2^{-2q-100}.
\end{equation}
Using a very similar procedure one can establish that
\begin{equation}\label{eq:diff_dir_est2}
    \Vert w_{q+1,k}^+Tw_{q+1,k'}^-\Vert_{\dot{H}^{-s}} < d^{-1}2^{-2q-100},
\end{equation}
\begin{equation}\label{eq:diff_dir_est3}
    \Vert w_{q+1,k}^-Tw_{q+1,k'}^+\Vert_{\dot{H}^{-s}} < d^{-1}2^{-2q-100},
\end{equation}
and
\begin{equation}\label{eq:diff_dir_est4}
    \Vert w_{q+1,k}^-Tw_{q+1,k'}^-\Vert_{\dot{H}^{-s}} < d^{-1}2^{-2q-100}.
\end{equation}
Thus from ~\eqref{eq:diff_dir_est1}-~\eqref{eq:diff_dir_est4} we obtain
\begin{equation}\label{eq:high_freq_est_final}
    \left\Vert \sum_{k \not = k'} w_{q+1,k}Tw_{q+1,k'}\right\Vert_{\dot{H}^{-s}} < 4\left(\frac{d-1}{d}\right)2^{-2q-100} < 2^{-2q-98}.
\end{equation}
And so from ~\eqref{eq:low_freq_est_final} and ~\eqref{eq:high_freq_est_final} we have
\begin{equation}\label{eq:osc_final_est}
    \Vert R_O \Vert_{\dot{H}^{-s}} < 2^{-2q-98} + 2^{-2q-96} < 2^{-2q-93}.
\end{equation}
And so finally from ~\eqref{eq:diss_final_est}, ~\eqref{eq:Nash_final_est}, and ~\eqref{eq:osc_final_est} we have
\begin{equation}\label{eq:Reynolds_final_est}
    \Vert R_{q+1} \Vert_{\dot{H}^{-s}} < 2^{-2q-93} + (2)2^{-2q-100} < 2^{-2q-92} < 2^{-q-1}.
\end{equation}
~\eqref{eq:Reynolds_final_est} establishes \ref{i:3}.

\subsection{Proof of Item \ref{i:4}}
From ~\eqref{eq:wq+1_est} we have
\begin{equation}\label{eq:L^p_est}
    \Vert \theta_{q+1} - \theta_q \Vert_{L^p} = \Vert w_{q+1} \Vert_{L^p} \leq \lambda_{q+1}^{(1-\epsilon_{q+1})\left(\frac{1}{2} - \frac{1}{p}\right)} < 2^{-C'(p)q}
\end{equation}
for $C'(p) = 1/p - 1/2$ and $\lambda_{q+1}$ chosen large enough. For the Besov norm estimate, we utilize the following variant of the classical Bernstein inequality.
\begin{lemma}\label{lem:Bernstein}
    Fix $\alpha \in \R$ and $\lambda > 0$ large. Suppose $u:\T^d \to \R$ is smooth and
    $\operatorname{supp}(\hat{u}) \subset \{\xi : |\xi| \simeq \lambda\}$. Then
    \begin{equation*}
        \Vert u \Vert_{\dot{B}^\alpha_{\infty,\infty}} \lesssim \lambda^\alpha \Vert u \Vert_{L^{\infty}}
    \end{equation*}
\end{lemma}
\begin{proof}
    See the proof of \cite[Lemma 5.7]{GR}.
\end{proof}

With this lemma in hand, we consider two cases. Suppose $q(\alpha)$ is the final index such that
$$
\alpha + \frac{1 - \epsilon_{q(\alpha)}}{2} \geq \frac{\alpha}{2}.
$$
If $q + 1 \leq q(\alpha)$, then we have
\begin{equation}\label{eq:Besov_est_1}
    \Vert \theta_{q+1} - \theta_q \Vert_{\dot{B}^\alpha_{\infty,\infty}} = \Vert w_{q+1} \Vert_{\dot{B}^\alpha_{\infty,\infty}} \leq \left(2^{q(\alpha)}\max_{0 \leq n \leq q(\alpha)} \Vert w_n \Vert_{\dot{B}^\alpha_{\infty,\infty}}\right)  2^{-q+1}
\end{equation}
We note the constant in ~\eqref{eq:Besov_est_1} depends only on the regularity parameter $\alpha$ and not $q$. On the other hand when $q + 1 > q(\alpha)$ then using Lemma \ref{lem:Bernstein} we have
\begin{equation}\label{eq:Besov_est_2}
    \Vert \theta_{q+1} - \theta_q \Vert_{\dot{B}^\alpha_{\infty,\infty}} = \Vert w_{q+1} \Vert_{\dot{B}^\alpha_{\infty,\infty}}\lesssim \lambda_{q+1}^{\alpha} \Vert w_{q+1} \Vert_{L^\infty} \lesssim \lambda_{q+1}^{\alpha + \frac{1-\epsilon_{q+1}}{2}} < 2^{-C''(\alpha)q}
\end{equation}
for $C''(\alpha) = -\alpha/2$ and $\lambda_{q+1}$ large enough. Hence from ~\eqref{eq:L^p_est}, ~\eqref{eq:Besov_est_1}, and ~\eqref{eq:Besov_est_2} we obtain that
\begin{equation*}
\begin{split}
    \Vert \theta_{q+1} - \theta_q \Vert_{\dot{B}^\alpha_{\infty,\infty}} + \Vert \theta_{q+1} - \theta_q \Vert_{L^p} &= 2^{-C'(p)q} + \left(2^{q(\alpha)+1} \max_{0 \leq n \leq q(\alpha)} \Vert w_{n} \Vert_{\dot{B}^\alpha_{\infty,\infty}}\right) 2^{-q} + 2^{-C''(\alpha)q}\\
    &< 2^{100}\left(2^{100} + 2^{\tilde{q}+1} \max_{0 \leq n \leq q(\alpha)} \Vert w_n \Vert_{\dot{B}^\alpha_{\infty,\infty}}\right) 2^{-C(\alpha,p)q}\\
    &= 2^{100 + C(\alpha,p)}\left(2^{100} + 2^{\tilde{q}+1} \max_{0 \leq n \leq q(\alpha)} \Vert w_n \Vert_{\dot{B}^\alpha_{\infty,\infty}}\right) 2^{-C(\alpha,p)(q+1)}\\
    \end{split}
\end{equation*}
where
$$
C(\alpha,p) = \min\left(C'(p),C''(\alpha)\right).
$$

\subsection{Proof of Item \ref{i:5}}
We assume that
$$
\Vert \theta_q \Vert_{L^1} > (1 + 2^{-q})\delta.
$$
Then we compute
$$
\Vert \theta_{q+1} \Vert_{L^1} \geq \Vert \theta_q \Vert_{L^1} - \Vert w_{q+1} \Vert_{L^1} > (1 + 2^{-q})\delta - \Vert w_{q+1} \Vert_{L^1}.
$$
Now choose $\lambda_{q+1}$ large enough such that
\begin{equation}\label{eq:delta_est}
    \Vert w_{q+1} \Vert_{L^1} < 2^{-q-1}\delta.
\end{equation}
This gives
$$
\Vert \theta_{q+1} \Vert_{L^1} > (1 + 2^{-q})\delta - 2^{-q-1}\delta = (1-2^{-q-1})\delta
$$

\subsection{Proof of Item \ref{i:6}}
Recall we had chosen $c$ to satisfy the estimates
$$
|k|^{-1} + r \leq c \leq \frac{12}{7}|k|^{-1} - r.
$$
Thus we have that
$$
\lambda_{q+1} \leq (c-r)|k|\lambda_{q+1} \leq (c+r)|k|\lambda_{q+1} \leq \frac{12}{7}\lambda_{q+1}.
$$
Therefore recalling Definition \ref{def:projs} and the fact $\lambda_{q+1}$ is a power of $2$ we have
$$
\mathbb{P}_{\lambda_{q+1}}(\theta_{q+1} - \theta_q) = \mathbb{P}_{\lambda_{q+1}}(w_{q+1}) = w_{q+1}
$$
since
$$
(c-r)|k|\lambda_{q+1} = |\sigma_{q+1}k - r\lambda_{q+1}k|
$$
and
$$
(c+r)|k|\lambda_{q+1} = |\sigma_{q+1}k + r\lambda_{q+1}k|
$$
implying that
$$
\operatorname{supp}(\hat{w}_{q+1}) \subset B(\sigma_{q+1}k,r\lambda_{q+1}) \subset \left\{\xi : \lambda_{q+1} \leq |\xi| \leq \frac{12}{7}\lambda_{q+1}\right\}.
$$

\subsection{Proof of Item \ref{i:7}}
We have that
\begin{equation}\label{eq:off_diag_1}
    \begin{split}
        \sum_{\substack{n,m \leq q+1\\ n \not  = m}} \Vert w_n Tw_m \Vert_{\dot{H}^{-s}}
        &= \sum_{\substack{n,m \leq q\\ n \not  = m}} \Vert w_n Tw_m \Vert_{\dot{H}^{-s}} + \sum_{n \leq q} \Vert w_n Tw_q \Vert_{\dot{H}^{-s}} + \sum_{m \leq q} \Vert w_q Tw_m \Vert_{\dot{H}^{-s}}\\
        &\lesssim C_1 - 2^{-q} + \sum_{n \leq q} \Vert w_n Tw_{q+1} \Vert_{L^1} + \sum_{m \leq q} \Vert w_{q+1} Tw_m \Vert_{L^1}
    \end{split}
\end{equation}
Now we clearly have
\begin{equation}\label{eq:para_est_1}
    \sum_{n \leq q} \Vert w_n Tw_{q+1} \Vert_{L^1} \lesssim \sum_{n \leq q} \Vert w_n \Vert_{L^4} \Vert Tw_{q+1} \Vert_{L^{4/3}} \lesssim \Vert w_{q+1} \Vert_{L^{4/3}} < 2^{-2q-100}
\end{equation}
and
\begin{equation}\label{eq:para_est_2}
    \sum_{m \leq q} \Vert w_{q+1} Tw_{m} \Vert_{L^1} \lesssim \sum_{m \leq q} \Vert w_{q+1} \Vert_{L^{4/3}} \Vert Tw_m \Vert_{L^{4}} \lesssim \Vert w_{q+1} \Vert_{L^{4/3}} < 2^{-2q-100}.
\end{equation}
for $\lambda_{q+1}$ chosen large enough. Thus from ~\eqref{eq:off_diag_1}, ~\eqref{eq:para_est_1}, and ~\eqref{eq:para_est_2} we get
\begin{equation}\label{eq:off_diag_final}
    \sum_{\substack{n,m \leq q+1\\ n \not  = m}} \Vert w_n Tw_m \Vert_{\dot{H}^{-s}} < C_1 - 2^{-q} + 2^{-2q-99} < C_1 - 2^{-q-1}.
\end{equation}
For the on diagonal estimate, we have using \ref{i:3} that
\begin{equation}\label{eq:off_diag_est}
\begin{split}
    \sum_{n \leq q+1} \Vert w_{n} Tw_{n} \Vert_{\dot{H}^{-s}} &= \sum_{n \leq q} \Vert w_{n} Tw_{n} \Vert_{\dot{H}^{-s}} + \Vert w_{q+1} Tw_{q+1} \Vert_{\dot{H}^{-s}}\\
    &< C_2 - 2^{-q+100} + \Vert R_q + w_{q+1} Tw_{q+1} \Vert_{\dot{H}^{-s}} + \Vert R_q \Vert_{\dot{H}^{-s}}\\
    &< C_2 - 2^{-q+100} + 2^{-q-1} + 2^{-q}\\
    &< C_2 - 2^{-q+99}.
    \end{split}
\end{equation}
~\eqref{eq:off_diag_final} and ~\eqref{eq:off_diag_est} combine to complete the proof of \ref{i:7} and therefore Proposition \ref{prop:ind} as well.

\noindent\textsc{Department of Mathematics, Purdue University, West Lafayette, IN, USA.}
\vspace{.03in}
\newline\noindent\textit{Email address}: \href{mailto:ngismond@purdue.edu}{ngismond@purdue.edu}.

\end{document}